\documentclass{amsart}
\usepackage{mathrsfs}
\usepackage{amssymb}
\usepackage{latexsym}
\usepackage{yfonts}
\usepackage{natbib}

\usepackage{graphicx,color}

\usepackage[pdftex,plainpages=false,colorlinks,hyperindex,bookmarksopen,linkcolor=red,citecolor=blue,urlcolor=blue]{hyperref}

\bibpunct{[}{]}{;}{n}{,}{,}

\newtheorem{te}{Theorem}
\newtheorem{defin}{Definition}
\newtheorem{os}{Remark}

\newtheorem{lem}{Lemma}
\newtheorem{coro}{Corollary}

\numberwithin{equation}{section}

\allowdisplaybreaks

\begin{document}

\title[Fractional advection diffusion equations]{Wright functions governed by fractional directional derivatives and fractional advection diffusion equations}

\author{Mirko D'Ovidio} 
\address{Dipartimento di Scienze di Base e Applicate per l'Ingegneria, Sapienza University of Rome}
\email{mirko.dovidio@uniroma1.it}


\date{\today}


\begin{abstract}
We consider fractional directional derivatives and establish some connection with stable densities. Solutions to  advection equations involving fractional directional derivatives are presented and some properties investigated. In particular we obtain solutions written in terms of Wright functions by exploiting operational rules involving the shift operator. We also consider fractional advection diffusion equations involving fractional powers of the negative Laplace operator and directional derivatives of fractional order and discuss the probabilistic interpretations of solutions.
\end{abstract}

\maketitle

\section{Introduction}
In this work we study the solutions to fractional Cauchy problems involving the operator
\begin{equation}
(\mathbf{a} \cdot \nabla) =  \sum_{k=1}^{n} a_{k}\, \partial_{x_k} \label{dirder}
\end{equation}
where $\nabla = (\partial_{x_1}, \ldots,  \partial_{x_n})$ and $\|\mathbf{a}\|^2=a_1^2 + \cdots + a_n^2=1$. In particular, we are interested in studying the solutions to the fractional advection equation
\begin{equation*} 
\frac{\partial^\beta \psi}{\partial t^\beta} + ( \mathbf{a}\cdot \nabla )^\alpha \psi = 0, \quad \alpha, \beta \in (0,1]
\end{equation*}
for the scalar field $\psi=\psi(\mathbf{x}, t)$ with constant velocity $\mathbf{a}$. By exploiting operational rules involving the shift operator, we obtain some representations of the scalar field $\psi$ in terms of Wright functions. From this, we arrive at the fractional advection (convection) diffusion equation
\begin{equation*}
\frac{\partial^\beta w}{\partial t^\beta} = -(-\triangle)^\vartheta w - (\mathbf{a}\cdot \nabla)^\alpha w, \quad \vartheta, \alpha, \beta \in (0,1] 
\end{equation*}
(where $\triangle = \nabla \cdot \nabla$ is the Laplace operator) and the transport equation
\begin{equation*}
\frac{\partial^\beta }{\partial t^\beta} \mathsf{w}= - (-\mathbf{a}\cdot \nabla)^\alpha  \mathsf{w} - \lambda (I - K)  \mathsf{w}
\end{equation*}
where $K$ is the Frobenius-Perron operator corresponding to some transformation. For the above equations we study the solutions from a probabilistic point of view.

Before starting with an overview of the work we introduce some notations which turn out to be useful further in the text:
\begin{itemize}
\item $u=u(\mathbf{x},t)$ is a general solution to a general boundary value problem,
\item $\widetilde{u}$ is the Laplace transform of $u$,
\item $\widehat{u}$ is the Fourier transform of $u$,
\item $\mathbf{a} \cdot \mathbf{x} = \sum_{k=1}^n a_k x_k$ and $\| \mathbf{x}\|^2 = \mathbf{x} \cdot \mathbf{x}$,
\item $\frac{\partial^\beta}{\partial t^\beta}$ is the Dzerbayshan-Caputo time-fractional derivative,
\item $\partial_t^\beta$ and $\partial_x^\alpha$ are the Riemann-Liouville time- and space-fractional derivatives,
\item $\partial_x=\partial / \partial x$ and $\partial_t=\partial /\partial t$,
\item $h_\alpha$ is the density law of the stable subordinator $\mathfrak{H}^{\alpha}_t$, $t>0$,
\item $l_\beta$ is the law of $\mathfrak{L}^{\beta}_t$, $t>0$ which is the inverse to $\mathfrak{H}^{\beta}$.
\end{itemize}

After some preliminaries and auxiliary results, we introduce, in Section \ref{secII}, the rule
\begin{equation}
e^{\zeta \partial_x} f(x) = f(x+\zeta) \label{opRuleIntro}
\end{equation}
where the shift operator $\exp \zeta \partial_x$ plays an important role in studying fractional powers of $\partial_x$ and \eqref{dirder}. By exploiting such a rule we define the fractional directional derivative
\begin{equation*}
(\mathbf{a}\cdot \nabla)^\alpha = \int_0^\infty  \left(1- e^{-s (\mathbf{a}\cdot \nabla)}\right) d\gamma(s)
\end{equation*}
where
\begin{equation*}
d\gamma(s)/ds = \frac{\alpha \, s^{-\alpha-1}}{\Gamma(1-\alpha)}
\end{equation*}
is the L\'evy measure of a stable subordinator. In Section \ref{secIII}, we study fractional equations of the form
\begin{equation}
\left( \frac{\partial^\beta}{\partial t^\beta} + (\mathbf{a}\cdot \nabla)^\alpha\right) u_{\alpha, \beta}(\mathbf{x}, t) = 0, \quad (\mathbf{x},t) \in \mathbb{R}^n_{+}\times (0, +\infty) \label{pdeIntro1}
\end{equation}
for $\alpha, \beta \in (0, 1]$. In particular we show that 
\begin{equation*}
u_{\alpha, \beta}(\mathbf{x}, t) = \mathcal{U}^{\alpha}_{\beta}(\mathbf{a}\cdot \mathbf{x}, t)
\end{equation*}
is the solution to the equation \eqref{pdeIntro1} subject to the initial and boundary conditions
\begin{equation*}
u_{\alpha, \beta}(\mathbf{x}, 0) = \delta(x_1) \times \cdots \times \delta(x_n), \quad u_{\alpha, \beta}(\mathbf{0}, t) = 0
\end{equation*}
where 
\begin{equation}
\mathcal{U}^{\alpha}_{\beta}(x, t) = \mathbb{E}^{x} \delta(\mathfrak{H}^\alpha_{\mathfrak{L}^\beta_t}), \quad x \in \mathbb{R},\; t>0 \label{funLamp}
\end{equation}
is the density law of the composition involving a stable subordinator $\mathfrak{H}^\alpha_t$ and an inverse process $\mathfrak{L}^\beta_t$. If $\alpha=\beta \in (0,1)$, then the function \eqref{funLamp} becomes the Lamperti's law (see for example \citet{Lanc, Dov2}) and the solution to \eqref{pdeIntro1}, for $\alpha = \beta$, can be explicitly written as
\begin{equation*}
u_{\beta, \beta}(\mathbf{x}, t) = \frac{\sin \beta \pi}{\pi} \frac{(\mathbf{a} \cdot \mathbf{x})^{\beta -1}\, t^{\beta} }{(\mathbf{a} \cdot \mathbf{x})^{2\beta} + 2 (\mathbf{a} \cdot \mathbf{x})^{\beta} t^\beta \cos \beta \pi + t^{2\beta}}, \quad \beta \neq 1.
\end{equation*}
This is the case in which the density law of the composition of processes appearing in \eqref{funLamp} coincides with the law of the ratio of two independent stable subordinators ${_j\mathfrak{H}^\beta_t}$, $j=1,2$, that is
\begin{equation}
\mathfrak{H}^\beta_{\mathfrak{L}^\beta_t} \stackrel{law}{=} t \times {_1\mathfrak{H}^\beta_t} / {_2\mathfrak{H}^\beta_t}, \quad t>0. \label{ratioSS}
\end{equation}

A special case of \eqref{pdeIntro1}, for $\beta \in (0,1)$, is the time-fractional equation
\begin{equation*}
\Big( \partial^\beta_t + (\mathbf{a}\cdot \nabla) \Big)v_\beta (\mathbf{x}, t)= 0, \quad (\mathbf{x},t) \in \mathbb{R}^n_+ \times (0,+\infty)
\end{equation*}
subject to the initial and boundary conditions
\begin{equation*}
v_\beta(\mathbf{x}, 0)=\delta(\mathbf{x}), \qquad v_\beta(\mathbf{0}, t)=t^{-n\beta}_{+}
\end{equation*}
for which we show that the solution can be written as
\begin{equation*}
v_\beta(\mathbf{x}, t) = \frac{1}{t^{n\beta}} W_{-\beta, 1-n \beta}\left(- \frac{\mathbf{a} \cdot \mathbf{x}}{t^{\beta}} \right)
\end{equation*}
where
\begin{align}
W_{\mu, \rho}(z) = & \frac{1}{2\pi i} \int_{Ha} \zeta^{-\rho} e^{\zeta + z\zeta^{-\mu}}d\zeta \notag \\ 
= & \sum_{k=0}^{\infty} \frac{z^k}{k!\, \Gamma(\mu k + \rho)}, \quad \mu >-1,\; \rho \in \mathbb{C}, \; z \in \mathbb{C} \label{wright}
\end{align}
($Ha$ denotes the Hankel path) is the Wright function which has been introduced and investigated by the British mathematician E. Maithland Wright in a series of paper starting from 1933. By normalizing $v_\beta$ we obtain the $n$-dimensional probability law $p_\beta$ whose one-dimensional marginals coincide with the density law of the hitting time $\mathfrak{L}^\beta_t = \inf \{ s\geq 0\,: \, \mathfrak{H}^\beta_s \notin (0,t) \}$, $t>0$ which is the inverse to the stable subordinator $\mathfrak{H}^\beta_t$, $t>0$. In particular, we get that
\begin{equation}
p_\beta(\mathbf{x}, t) = \frac{Pr\{ {_1X^\beta_t} \in dx_1, \ldots ,{_nX^\beta_t} \in dx_n \}}{dx_1\cdots dx_n}
\end{equation}
is the law of the process $\mathscr{X}_\beta(t) = ({_1X^\beta_t}, \ldots ,{_nX^\beta_t})$ which can be regarded as the inverse of
\begin{align*}
\mathscr{H}_\beta(\mathbf{x})  & = \sqrt[\beta]{a_1}\times {_1 \mathfrak{H}^\beta_{x_1}}+ \ldots + \sqrt[\beta]{a_n}\times {_n \mathfrak{H}^\beta_{x_n}}
\end{align*}
in the sense that
\begin{align*}
Pr\{ \mathscr{X}_\beta(t) < \mathbf{x} \} = Pr\{ \mathscr{H}_\beta(\mathbf{x}) >t \}.
\end{align*}
Due to the fact that $\mathfrak{H}^\beta_t$, $t>0$, has non-negative increments, the multi-parameter process $\mathscr{H}_\beta(\mathbf{x})$, $\mathbf{x}\in \mathbb{R}^n_+$, possesses non-decreasing paths.

For the time-fractional equation
\begin{equation*}
\Big( \partial^\beta_t + (\mathbf{a}\cdot \nabla) \Big) \mathfrak{U}_\beta^{n}(\mathbf{x}, t) = 0, \quad \beta \in (0,1), \; (\mathbf{x},t) \in \mathbb{R}^n_{+} \times (0,+\infty)  
\end{equation*}
subject to the initial and boundary conditions
\begin{equation*}
\mathfrak{U}_\beta^{n}(\mathbf{x}, 0)=\delta(\mathbf{x}), \qquad \mathfrak{U}_\beta^{n}(\mathbf{0}, t)=t^{n - n\beta -1}_{+}, \quad n \in \mathbb{N}
\end{equation*}
we obtain that
\begin{equation*}
\mathfrak{U}_\beta^{n}(\mathbf{x}, t) = t^{n-n\beta -1}W_{-\beta, n-n\beta}\left( -\frac{\mathbf{a}\cdot \mathbf{x}}{t^{\beta}} \right)
\end{equation*}
and, for $n,m \in \mathbb{N}$, we show that
\begin{equation*}
\mathfrak{U}_\beta^{n}(\mathbf{x}, t) * \mathfrak{U}_\beta^{m}(\mathbf{y}, t) = \mathfrak{U}_\beta^{n+m}(\mathbf{x} + \mathbf{y}, t)
\end{equation*}
where $*$ stands for the Laplace convolution over $t$. Furthermore, for $n=1$, we obtain the density law of $\mathfrak{L}^\beta_t$, $t>0$.

In Section \ref{secIV}, we present the solutions to the time-fractional problems involving the operator
\begin{equation} 
|\mathbf{a} \cdot \nabla |^{\alpha} = \left( (\mathbf{a} \cdot \nabla)^2 \right)^{\frac{\alpha}{2}} =\left( \sum_{i=0}^{n} \sum_{j=0}^n a_{ij}\frac{\partial^2}{\partial x_i \partial x_j} \right)^{\frac{\alpha}{2}}
\end{equation} 
and, for $\alpha=2$, we find out some connection with Gaussian laws. In particular, we show that the solution to
\begin{equation*} 
\left( \frac{\partial^\beta}{\partial t^\beta} - (\mathbf{a} \cdot \nabla)^2 \right) g(\mathbf{x}, t)=0, \quad  \beta \in (0,1), \;  (\mathbf{x},t) \in \mathbb{R}^n \times (0, +\infty)
\end{equation*}
subject to the initial condition $g(\mathbf{x}, 0)=\delta(\mathbf{x})$ is written in terms of the Wright function \eqref{wright} as follows
\begin{equation}
g(\mathbf{x}, t) = \frac{1}{t^{\beta/2}}W_{-\frac{\beta}{2}, 1- \frac{\beta}{2}} \left(- \frac{|\mathbf{a} \cdot \mathbf{x}|}{t^{\beta/2}} \right). \label{sol2Ordintro}
\end{equation}
Furthermore, formula \eqref{sol2Ordintro} can be written as
\begin{equation}
g(\mathbf{x}, t) = \int_0^\infty \frac{e^{-\frac{(\mathbf{a} \cdot \mathbf{x})^2}{4s}}}{\sqrt{4\pi s}}l_{\beta}(s, t)ds 
\end{equation}
where
\begin{equation}
l_\beta(s, t) = \frac{1}{t^\beta} W_{-\beta, 1-\beta}\left(- \frac{s}{t^\beta} \right) \label{llaw}
\end{equation}
is the law of the inverse process $\mathfrak{L}^\beta_t$, $t>0$. 

Finally, we arrive at the fractional advection diffusion equation. We show that, for $\vartheta \in (0,1)$, $\alpha, \beta \in (0,1)$ and $\mathbf{a}\in \mathbb{R}^n$ such that $\| \mathbf{a}\| =1$, the solution $w=w(\mathbf{x},t)$ to the fractional equation
\begin{equation*}
\frac{\partial^\beta w}{\partial t^\beta} = -(- \triangle)^\vartheta w - (\mathbf{a}\cdot \nabla)^\alpha w
\end{equation*}
for $(\mathbf{x}, t) \in \mathbb{R}^n \times (0, +\infty)$, is given by
\begin{equation*}
w(\mathbf{x}, t) = \int_0^\infty dz \int_{\mathbb{R}^n_+} d\mathbf{s} \,\mathcal{T}_{2\vartheta}(\mathbf{x} - \mathbf{s}, z) \, h_\alpha(\mathbf{a}\cdot \mathbf{s}, z)\, l_\beta(z, t) 
\end{equation*}
where $\mathcal{T}_{2\vartheta}(\mathbf{x}, t)$  is the law of the isotropic stable L\'evy process $\mathbf{S}_{2\vartheta}(t)$, $h_\alpha(x, t)$ is the law of the stable subordinator $\mathfrak{H}^\alpha_t$ and $l_\beta(x, t)$ is the law of the inverse process $\mathfrak{L}^\beta_t$.  The distribution $w$ represents the density law of a subordinated $\mathbb{R}^n$-valued stable process with stable subordinated  drift given by
\begin{equation*}
\mathbf{W}(t) = \mathbf{S}_{2\vartheta}(\mathfrak{L}^\beta_t) + \mathbf{a}\, \mathfrak{H}^\alpha_{\mathfrak{L}^\beta_t}, \quad t>0.
\end{equation*} 
We also show that the stochastic solution to the transport equation 
\begin{equation}
\frac{\partial^\beta \mathsf{w}}{\partial t^\beta} = - (\mathbf{a}\cdot \nabla)^\alpha \mathsf{w} - \lambda \left( I - K \right) \mathsf{w}
\end{equation}
where $\alpha, \beta \in (0, 1]$ and $K=e^{-\mathbf{1}\cdot \nabla}$ is the shift operator, is given by
\begin{equation*}
\mathbf{Y}_t = \mathbf{N}(\mathfrak{L}^\beta_t) + \mathbf{a}\, \mathfrak{H}^\alpha_{\mathfrak{L}^\beta_t}.
\end{equation*}
where $\mathbf{N}_t$ is a Poisson process.  If $\mathbf{a}=\mathbf{0}$, then $K=B$ becomes the backward operator and therefore $\mathsf{w}$ becomes the probability of the fractional Poisson process $\mathbf{N}(\mathfrak{L}^\beta_t)$ (see for example \citet{OB09EJP} and the references therein).

\section{Preliminaries and auxiliary results}
For $m-1 < \alpha < m$, $m \in \mathbb{N}$, we define the Weyl's fractional derivative
\begin{equation*}
\partial_x^\alpha f(x) = \partial_x^m \left[ \frac{x^{m-\alpha-1}}{\Gamma(m-\alpha)} * f(x) \right]= \frac{\partial_x^m}{\Gamma(m-\alpha)} \int_{-\infty}^x \frac{f(s)\, ds}{(x-s)^{\alpha + m-1}}, \quad x\in \mathbb{R}
\end{equation*}
which is named Riemann-Liouville derivative if $x>0$ whereas, for $m-1 < \beta < m$, we define the Dzerbayshan-Caputo fractional derivative as
\begin{equation*}
\frac{\partial^\beta f}{\partial t^\beta}(t) = \left[\frac{t^{m-\beta-1}}{\Gamma(m-\beta)} * \partial_t^m f(t)\right] = \frac{1}{\Gamma(m-\beta)} \int_{0}^{t} \frac{\partial_s^m f(s)\, ds}{(t-s)^{\alpha +m -1}}, \quad t>0.
\end{equation*}
The symbol $*$ stands for the Fourier convolution and for $\alpha=\beta=1$, both fractional derivatives become ordinary derivatives. Furthermore, we recall the following connection between the above fractional derivatives (\citep{SKM93})
\begin{equation}
\frac{\partial^\beta f}{\partial t^\beta}(t) = \partial^\beta_t f(t) - \sum_{k=0}^{m-1} \partial^k_t f(t)\Big|_{t=0^+}\, \frac{t^{k-\beta}}{\Gamma(k-\beta +1)} \label{connectionFD}
\end{equation}
which leads to the following relation between Laplace transforms
\begin{equation*}
\widetilde{\frac{\partial^\beta f}{\partial t^\beta}}(\lambda) = \widetilde{\partial^\beta_t f}(\lambda) - \sum_{k=0}^{m-1}  \partial^k_t f(t)\Big|_{t=0^+}\, \lambda^{\beta -k -1}
\end{equation*}
where $\widetilde{\partial^\beta_t f}(\lambda) = \lambda^\beta \widetilde{f}(\lambda)$.

Our aim in this section is to introduce the operational solutions to the fractional equation
\begin{equation} 
\left( \frac{\partial^\beta}{\partial t^\beta}  + \partial_x^\alpha\right) \mathcal{U}^{\alpha}_{\beta} = 0, \quad x \geq 0,\, t>0
\label{eqPrelim1}
\end{equation}
subject to the initial and boundary conditions $\mathcal{U}^{\alpha}_{\beta}(x, 0) = \varphi_0(x)$ and $\mathcal{U}^{\alpha}_{\beta}(0, t)=0$ with $\alpha, \beta \in (0, 1]$. The problem  is to find an explicit form for the convolution
\begin{equation}
\mathcal{U}^{\alpha}_{\beta}(x, t) = E_{\beta}(-t^\beta \partial_x^\alpha)\varphi_0(x) \label{explCOnv}
\end{equation}
where
\begin{equation}
E_{\beta}(z) = \frac{1}{2\pi i}\int_{Ha} \frac{\zeta^{\beta -1} e^\zeta}{\zeta^\beta -z}d\zeta =  \sum_{k=0}^\infty \frac{z^k}{\Gamma(\beta k +1)}, \quad \Re\{ \beta \} >0, \; z \in \mathbb{C} \label{mittag-leffler}
\end{equation}
($Ha$ is the Hankel path) is the Mittag-Leffler function. We first observe that $u(z)=E_\beta(-w^\beta z)H(z)$, where $w>0$ and $H(z)$ is the Heaviside step function, is the fundamental solution to the fractional relaxation equation
\begin{equation}
\frac{\partial^\beta u}{\partial z^\beta}(z) + w^\beta \, u(z)=0.\label{relaxEq}
\end{equation}
For $\varphi_0=\delta$, we get $\widetilde{\varphi_0}(\xi) = 1$ and thus, being $\widetilde{\partial_x^\alpha \mathcal{U}^{\alpha}_{\beta}}(\xi, t)=\xi^\alpha \widetilde{\mathcal{U}^{\alpha}_{\beta}}(\xi, t)$ the Laplace transform of the Riemann-Liouville derivative of $\mathcal{U}^{\alpha}_{\beta}$, we obtain the Laplace transform
\begin{equation}
\widetilde{\mathcal{U}^{\alpha}_{\beta}}(\xi, t) = E_\beta(-t^\beta \xi^\alpha) .\label{lapeqPrelim1}
\end{equation}
As we can immediately check, from \eqref{relaxEq} or the well-known fact that $E_\beta$ is an eigenfunction for the Dzerbayshan-Caputo derivative, we reobtain
\begin{equation*}
\frac{\partial^\beta}{\partial t^\beta} \widetilde{\mathcal{U}^{\alpha}_{\beta}}(\xi, t) = -\xi^\alpha \, \widetilde{\mathcal{U}^{\alpha}_{\beta}}(\xi, t) = \int_0^\infty e^{-\xi x} \left( - \partial^\alpha_x \mathcal{U}^{\alpha}_{\beta}(x, t) \right) dx.
\end{equation*}
From \eqref{lapeqPrelim1} and the fact that (see for example \citet{SKM93})
\begin{equation*} E_\beta(-z^\beta a^\alpha) = \frac{\sin \beta \pi}{\pi} \int_0^\infty \frac{w^{\beta -1} \,e^{- w z a^{\alpha/\beta}} \, dw}{w^{2\beta} + 2 w^{\beta} \cos \beta \pi + 1}, \quad a>0,
\end{equation*}
we can write
\begin{align}
\widetilde{\mathcal{U}^{\alpha}_{\beta}}(\xi, t) = & \frac{\sin \beta \pi}{\pi}  \int_0^\infty e^{- x \xi^{\alpha / \beta}} \frac{x^{\beta -1}\, t^{\beta}}{x^{2\beta} + 2 x^{\beta} t^\beta \cos \beta \pi + t^{2\beta}} dx. \label{aboveCML}
\end{align}
If we formally rearrange \eqref{explCOnv} from \eqref{aboveCML} we get that
\begin{equation}
\mathcal{U}^{\alpha}_{\beta}(x, t) =  \frac{\sin \beta \pi}{\pi}  \int_0^\infty ds\,  \frac{s^{\beta -1}\, t^{\beta}}{s^{2\beta} + 2 s^{\beta} t^\beta \cos \beta \pi + t^{2\beta}} \, e^{- s \partial_x^{\alpha / \beta}} \varphi_0(x) \label{expliciRep}
\end{equation}
which takes, for $\alpha=\beta$, the following form
\begin{equation}
\mathcal{U}^{\beta}_{\beta}(x, t) =  \frac{\sin \beta \pi}{\pi}  \int_0^\infty ds\,  \frac{s^{\beta -1}\, t^{\beta}}{s^{2\beta} + 2 s^{\beta} t^\beta \cos \beta \pi + t^{2\beta}} \varphi_0(x-s) \label{explicit22}
\end{equation}
by taking into account the rule \eqref{opRuleIntro}. Thus, from \eqref{explicit22},  for $\varphi_0=\delta$, we arrive at 
\begin{equation*}
\mathcal{U}_{\beta}^{\beta}(x, t) = \frac{\sin \beta \pi}{\pi} \frac{x^{\beta -1}\, t^{\beta} }{x^{2\beta} + 2 x^{\beta} t^\beta \cos \beta \pi + t^{2\beta}}
\end{equation*}
which is the solution to the equation \eqref{eqPrelim1} or the inverse Laplace transform of \eqref{aboveCML} for $\alpha=\beta$. For a well-defined function $\varphi_0$, the explicit representation of the convolution \eqref{explCOnv} is therefore given by \eqref{expliciRep} or, for $\alpha=\beta$, by \eqref{explicit22}.

The solution $ \mathcal{U}^{\alpha}_{\beta}(\xi, t)$, $x \geq 0$, $t>0$ for $\alpha, \beta \in (0,1)$, can be regarded as the law of the composition $\mathfrak{H}^\alpha_{\mathfrak{L}^\beta_t}$, $t>0$, where $\mathfrak{H}^\alpha_t$, $t>0$ is an $\alpha$-stable subordinator (see \citet{Btoi96}) with Laplace transform
\begin{equation}
\mathbb{E} \exp - \xi \mathfrak{H}^\alpha_t = \exp - t \xi^\alpha \label{lapHspace}
\end{equation}
and $\mathfrak{L}^\beta_t$, $t>0$, is the inverse to the $\beta$-stable subordinator $\mathfrak{H}^\beta_t$ (see for example \citet{BMN09}) for which
\begin{equation}
\mathbb{E} \exp - \xi \mathfrak{L}^\beta_t = E_\beta(- t^\beta \xi). \label{lapLspace}
\end{equation}
We also recall that
\begin{equation}
\mathbb{E}\exp i \xi \mathfrak{H}^\alpha_t = \exp\left( -t (-i \xi)^\alpha \right) = \exp\left( -t |\xi |^\alpha e^{- i \frac{\pi \alpha}{2}\frac{\xi}{|\xi |}} \right) \label{furHPrel}
\end{equation}
is the characteristic function of a totally (positively) skewed stable process. The stable subordinator $\mathfrak{H}^\alpha_t$ is a L\'evy process with non-negative, independent and stationary increments whereas, the inverse process $\mathfrak{L}^\beta_t$ has non-negative, non-independent and non-stationary increments as pointed out in \citet{MSheff04} and turns out to be fundamental in studying the solutions to some time-fractional  equations. We refer to $\mathfrak{L}^\alpha_t$ as the inverse of $\mathfrak{H}^\alpha_t$ in the sense that $P\{\mathfrak{L}^\alpha_t < x \} = P\{ \mathfrak{H}^\alpha_x > t\}$. We notice that, from the fact that $\mathfrak{H}^\alpha_t$ has non negative increments, the process $\mathfrak{L}^\alpha_t = \inf \{s\geq 0\, :\, \mathfrak{H}^\alpha_s \notin (0, t)\}$ is an hitting time. Let us write $h_\alpha$ for the law of $\mathfrak{H}^\alpha_t$ and $l_\beta$ for the law of $\mathfrak{L}^\beta_t$. The governing equations are known to be
\begin{equation}
\left( \frac{\partial}{\partial t} + \partial^\alpha_x \right) h_\alpha(x, t) = 0, \quad x>0,\; t>0
\end{equation}
subject to the initial and boundary conditions
\begin{equation*}
\left\lbrace \begin{array}{l}
h_\alpha(x, 0)= \delta(x)\\
h_\alpha(0, t) = 0
\end{array} \right.
\end{equation*}
and
\begin{equation}
\left( \partial^\beta_t + \frac{\partial}{\partial x}  \right)l_\beta(x, t)=0, \quad x>0,\; t>0 \label{eqlPrel}
\end{equation}
subject to 
\begin{equation}
\left\lbrace \begin{array}{l}
l_\alpha(x, 0)= \delta(x)\\
l_\alpha(0, t) = \frac{t^{-\beta}}{\Gamma(1-\beta)}.
\end{array} \right.
\label{condeqlPrel}
\end{equation}
In view of \eqref{connectionFD}, formula \eqref{eqlPrel} with conditions \eqref{condeqlPrel} can be rewritten in the compact form 
\begin{equation}
\left( \frac{\partial^\beta}{\partial t^\beta} +\frac{\partial}{\partial x} \right) l_\beta(x, t) = 0, \quad x\geq 0, \; t>0.
\end{equation}
From \eqref{relaxEq}, by considering that $\widetilde{\partial_x l_{\beta}}(\xi, t) = \xi\, \widetilde{l_\beta}(\xi, t)$, we obtain that $\widetilde{l_\beta}(\xi, t) = E_\beta(-t^\beta \xi)$ which is the Laplace transform \eqref{lapLspace}. 
We notice that  (see \citet{Dov4})
\begin{equation*}
\frac{t\, h_\alpha(t, x)}{x\, l_\alpha(x, t)} = \frac{x\, h_\alpha(x, t)}{t\, l_\alpha(t, x)} = \alpha \in (0,1), \quad x, t>0
\end{equation*}
where $l_\alpha$ can be written in terms of the Wright function \eqref{wright} as in \eqref{llaw}. Thus, we can write the solution to \eqref{eqPrelim1} as 
\begin{equation}
\mathcal{U}^{\alpha}_{\beta}(x, t) = \int_0^\infty h_\alpha (x, s) l_\beta(s, t)ds = \frac{\alpha}{x} \int_0^\infty s\,  l_\alpha(s, x) \, l_\beta(s, t) \, ds \label{funLampII}
\end{equation}
and, from \eqref{lapHspace} and \eqref{lapLspace}, we can write
\begin{equation*}
\widetilde{\mathcal{U}^{\alpha}_{\beta}}(\xi, t) = \int_0^\infty e^{-s \xi^\alpha } l_{\beta}(s, t)\, ds = E_\beta (-t^\beta \xi ^\alpha)
\end{equation*}
which is in accord with \eqref{lapeqPrelim1}. 

From the discussion above we arrive at the following fact
\begin{equation} 
E_\beta (-t^\beta \mathcal{A}^\alpha) =  \int_0^\infty ds\; \mathcal{U}^{\alpha}_{\beta}(s, t) \, e^{-s \mathcal{A}} \label{explCOnvII}
\end{equation}
which holds for some well-defined operator $\mathcal{A}$. In particular, for $\mathcal{A}=\partial_x$, from \eqref{explCOnvII} we write \eqref{explCOnv}. 

We also introduce the homogeneous distribution
\begin{equation}
z^\eta_{+} = \frac{z^\eta}{\Gamma(1+\eta)} H(z), \quad z \in \mathbb{R}, \; \Re\{ \eta \} >-1 \label{distrZ}
\end{equation}
(where $H$ is the Heaviside function) which is locally integrable in $\mathbb{R}\setminus \{0\}$. The Fourier transform of \eqref{distrZ}  is written as 
\begin{equation}
\int_{-\infty}^{+\infty} e^{i \xi z} z^\eta_{+}\, dz = |\xi |^{-\eta -1} e^{i \frac{\pi}{2}\frac{\xi}{|\xi |} (\eta +1)} \label{FurdistrZ}
\end{equation}
and can be obtained by considering the fact that
\begin{equation*}
z^\eta_{+} = \lim_{s \to 0^+} \left( e^{-sz} z^\eta_{+} \right).
\end{equation*}
Furthermore, as a straightforward calculation shows
\begin{equation}
\int_0^\infty e^{-\lambda z} z^\eta_{+} dz = \frac{1}{\lambda^{\eta +1}}. \label{LapdistrZ}
\end{equation}

\section{Fractional Directional Derivatives}
\label{secII}
Let $(X, d, \mu)$ be a metric space, that is $(X,d)$ is a locally compact separable metric space and $\mu$ is a Radon measure supported on $X$. Let $L^p(\mu)$ be the space of all real-valued functions on $X$ with finite norm
\begin{equation}
\| u \|_p^p := \int_X |u(x)|^p d\mu(x), \quad 1\leq p < \infty.
\end{equation}
Let $T_t$ be a strongly continuous semigroup on $L^p(\mu)$ and $\mathcal{A}$ be the infinitesimal generator of $T_t$, that is
\begin{equation}
\lim_{t\to 0^+} \Big\| \frac{T_t u - u}{t} - \mathcal{A}u \Big\|_p = 0 \label{existnormA}
\end{equation}
for all $u \in D(\mathcal{A})$ where
\begin{equation*}
Dom(\mathcal{A}) = \{ u \in L_p(\mu)\, : \, \textrm{the limit } \eqref{existnormA} \textrm{ exists} \}.
\end{equation*} 
We formally write
\begin{equation}
T_t = e^{t\mathcal{A}}, \quad t>0 \label{Tsemigroup}
\end{equation}
and $T^\alpha_t$ will denote the semigroup for the fractional generator $\mathcal{A}^\alpha$.

The fractional power of operators has been examined in many papers since the beginning of the Sixties. A power $\alpha$ of a closed linear operator $\mathcal{A}$ can be represented by means of the Dunford integral
\begin{equation}
\mathcal{A}^\alpha = \frac{1}{2\pi i} \int_\Gamma d\lambda \, \lambda^\alpha \, (\lambda - \mathcal{A})^{-1}, \quad \Re\{ \alpha \} >0
\label{AalphaDunf}
\end{equation}
under the conditions
\begin{equation*}
\begin{array}{rl} (i) & \lambda \in \rho(\mathcal{A}) \, (\textrm{the resolvent set of  } \mathcal{A}) \textrm{ for all } \lambda >0;\\
(ii) & \| \lambda (\lambda I + \mathcal{A})^{-1} \| < M < \infty \textrm{ for all }\lambda >0  \end{array}
\end{equation*}
where $\Gamma$ encircles the spectrum $\sigma(\mathcal{A})$ counter-clockwise avoiding the negative real axis and $\lambda^\alpha$ takes the principal branch. The reader can find information on this topic in works by \citet{Balak60}, \citet{HoWe72}, \citet{Kom66}, \citet{KraSob59}, \citet{Wata61}. For $\Re\{\alpha\} \in (0,1)$, the integral \eqref{AalphaDunf} can be rewritten in the Bochner sense as follows
\begin{equation}
\mathcal{A}^\alpha = \frac{\sin \pi \alpha}{\pi} \int_0^\infty d\lambda\, \lambda^{\alpha-1} (\lambda + \mathcal{A})^{-1} \mathcal{A}. \label{AalphaBoch}
\end{equation}
From the proof of the Hille-Yosida theorem we have that 
\begin{align*}
(\lambda + \mathcal{A})^{-1} = \int_0^\infty dt \, e^{-\lambda t} e^{-t \mathcal{A}}.
\end{align*}
By inserting this expression into \eqref{AalphaBoch} we get that
\begin{align*}
\int_0^\infty d\lambda\, \lambda^{\alpha-1} (\lambda + \mathcal{A})^{-1} = \left( \int_0^\infty ds\, s^{-\alpha} e^{-s\mathcal{A}}  \right) \left( \int_0^\infty ds\, s^{\alpha -1} e^{-s} \right)
\end{align*}
where
\begin{equation*}
\int_0^\infty ds\, s^{\alpha -1} e^{-s} = \Gamma(\alpha) 
\end{equation*}
is the Gamma function for which the following well-known property holds true
\begin{equation*}
\Gamma(\alpha)\Gamma(1-\alpha) = \frac{\pi}{\sin \pi \alpha}
\end{equation*}
and
\begin{equation}
\int_0^\infty ds\, s^{-\alpha} e^{-s \mathcal{A}}  = \Gamma(1-\alpha) \mathcal{A}^{\alpha -1} \label{powA}
\end{equation}
which holds only if $0 < \alpha < 1$ (the reader can also consults the interesting book by \citet{RenRogBook}). The representation \eqref{AalphaBoch} can be therefore rewritten as
\begin{equation*}
\mathcal{A}^\alpha = \mathcal{A}^{\alpha - 1} \mathcal{A}, \quad \alpha \in (0,1)
\end{equation*} 
by considering \eqref{powA}. After some algebra, by exploiting the same calculation, we can also arrive at the representation
\begin{equation}
\mathcal{A}^\alpha = \mathcal{A}^n \mathcal{A}^{\alpha -n}, \quad n-1 < \alpha < n, \; n \in \mathbb{N}.\label{argumentAA}
\end{equation}
The semigroup $T_t$ of the translation operator
$$\mathcal{A}=\frac{\partial}{\partial x}$$
is strongly continuous in $L_p(\mu)$ with $d\mu(x)/dx = e^{-\omega x}$, $\omega >0$ and $\|T_t \|_p = e^{-\omega t /p}$ (see for example \cite{ButWest2000, SKM93}). In what follows we will study the fractional power of the directional derivative operator $\mathcal{A}=(\mathbf{a}\cdot \nabla)$. 

The directional derivative of fractional order $\alpha \in (0,1)\cup (1,2]$ has been first introduced by \citet{MBB99} in connection with models of anomalous diffusions. By using the Fourier transform convention the authors define the fractional derivative $\mathbb{D}^\alpha_M$ of a well defined function $f:\mathbb{R}^n \mapsto \mathbb{R}$ by requiring that
\begin{equation}
\widehat{\mathbb{D}^\alpha_M f}(\boldsymbol{\xi}) = \widehat{f}(\boldsymbol{\xi})\, \int_{\| \boldsymbol{\theta} \| =1} (-i\boldsymbol{\xi} \cdot \boldsymbol{\theta})^\alpha M(d\boldsymbol{\theta}), \quad \boldsymbol{\xi} \in \mathbb{R}^n, \; \alpha \in (0,1) \cup (1, 2] \label{furMeer}
\end{equation}
where $M(d\boldsymbol{\theta})$ is a probability measure on the unit sphere 
$$\mathbb{S}^{n-1} = \{ \boldsymbol{\theta} \in \mathbb{R}^n: \; \|\boldsymbol{\theta} \| =1 \}.$$
In the one-dimensional case, $\theta=\pm 1$ so that $\| \theta \| = 1$. Formula \eqref{furMeer} becomes
\begin{equation*}
\widehat{f}(\xi) \left( p(-i\xi)^\alpha + q(i\xi)^\alpha \right), \quad \xi \in \mathbb{R}
\end{equation*}
with $p+q=1$ and therefore the inverse Fourier transform yelds the governing operator of a one-dimensional asymmetric stable process (or symmetric if $q=1/2$). 

In order to define fractional directional derivatives and study fractional equations involving the operator \eqref{dirder} we first introduce the rule 
\begin{equation}
e^{\zeta \partial_x}f(x) = f(x+\zeta) \label{opRule1}
\end{equation}
for an analytic function $f : \mathbb{R} \mapsto \mathbb{R}$ and $\zeta \in \mathbb{R}$. The operational rules \eqref{opRule1} plays an important role in quantum mechanics (\cite{BDatQ11}) and formally, can be obtained by considering the Taylor expansion of the analytic function $f$ near $x$  written as follows
\begin{equation*}
f(\zeta) = \sum_{k=0}^{\infty} \frac{(\zeta -x)^k}{k!} f^{(k)}(x) .
\end{equation*}
Thus, we get that
\begin{equation*}
f(x + \zeta) = \sum_{k=0}^{\infty}  \frac{\zeta^k}{k!} f^{(k)}(x) = \sum_{k=0}^{\infty}  \frac{\zeta^k}{k!}\partial^k_x f(x) = e^{\zeta \partial_x}f(x).
\end{equation*}
However (\citet{Fel71, HillePhil57}), the Taylor series and therefore the rule \eqref{opRule1}, can be extended to the class of bounded continuous functions on $(0,+\infty)$. The reader can find further details on the Taylor expansion for generalized functions in works by \citet{EstKan93, Anat09} and the references therein. 

In our view, from \eqref{argumentAA}, \eqref{powA} and \eqref{Tsemigroup} the fractional power of $\mathcal{A}$ can be written as follows
\begin{equation}
\mathcal{A}^\alpha = \mathcal{A} \mathcal{A}^{\alpha -1} = \mathcal{A} \left[ \frac{1}{\Gamma(1-\alpha)} \int_0^\infty ds\, s^{-\alpha} T_s \right]
\end{equation}
and therefore, based on the discussion made so far, we introduce the following representation of the fractional power of the directional derivative \eqref{dirder}.
\begin{defin}
For $0 < \alpha < 1$, $\mathbf{a} \in \mathbb{R}^n_{+}$ such that $\|\mathbf{a}\|=1$ and $\nabla=(\partial_{x_1}, \ldots , \partial_{x_n})$ we define
\begin{equation}
(\mathbf{a}\cdot \nabla)^\alpha := \frac{(\mathbf{a}\cdot \nabla)}{\Gamma(1-\alpha)} \int_{0}^\infty ds\, s^{-\alpha} e^{-s (\mathbf{a}\cdot \nabla)}. \label{deffracdirder}
\end{equation}
\end{defin}

\begin{os}
\normalfont
We notice that, after some algebra, from \eqref{deffracdirder} we get 
\begin{equation}
(\mathbf{a}\cdot \nabla)^\alpha = \int_0^\infty  \left(1- e^{-s (\mathbf{a}\cdot \nabla)}\right)d\gamma(s) \label{deffracdirder2}
\end{equation}
where
\begin{equation*}
d\gamma(s)/ds = \frac{\alpha \, s^{-\alpha-1}}{\Gamma(1-\alpha)}
\end{equation*}
is the L\'evy measure of a stable subordinator (\citet{Btoi96}). Indeed, a Berstein function $b:(0, +\infty) \mapsto \mathbb{R}$ and therefore such that $b\geq 0$ and $(-1)^k \partial^k_x b(x) \leq 0$ for all $x>0$ and $k \in \mathbb{N}$, is a Berstein function if and only if
\begin{equation}
b(x) = c_1 + c_2 x + \int_0^\infty (1-e^{sx})dm(s), \quad x>0
\end{equation}
for constants $c_1, c_2\geq 0$ and a non-negative measure $m$ on $[0, \infty)$ satisfying
\begin{equation}
\int_0^\infty (s \wedge 1)dm(s) < \infty.
\end{equation}
Furthermore, $b$ is a Bernstein function if and only if there is a convolution semigroup, say $\phi_t$, such that
\begin{equation}
\int_0^\infty e^{-sx} d\phi_t(s) = e^{-t b(x)}.
\end{equation} 
Here, we consider $b(x)=x^\alpha$ with $c_1=c_2=0$ and thus, $m=\gamma$ and $\phi_t(s) = h_\alpha(s, t)$ is the law of $\mathfrak{H}^\alpha_t$. 
\end{os}

\begin{os}
\normalfont
In the one-dimensional case, from  \eqref{deffracdirder}, for $f:(0, +\infty) \mapsto \mathbb{R}$, we obtain that the fractional derivative
\begin{align*}
\partial_x^\alpha f (x) = & \frac{\partial}{\partial x} \left[ \frac{1}{\Gamma(1-\alpha)} \int_0^\infty ds\, s^{-\alpha} T_sf(x) \right]
\end{align*}
where $-\partial_x$ is the generator of $T_s$, takes the form
\begin{align*}
\partial_x^\alpha f (x) = & \frac{\partial}{\partial x} \left[ \frac{1}{\Gamma(1-\alpha)} \int_0^{\infty} ds\, s^{-\alpha} e^{-s \partial_x}f(x) \right]\\
= & \frac{\partial}{\partial x} \left[ \frac{1}{\Gamma(1-\alpha)} \int_0^\infty ds\, s^{-\alpha} f(x-s) \right]\\
= & \frac{\partial}{\partial x} \left[ \frac{1}{\Gamma(1-\alpha)} \int_0^x ds\, s^{-\alpha} f(x-s) \right]\\
= & \frac{1}{\Gamma(1-\alpha)} \frac{\partial}{\partial x} \int_{0}^x \frac{ds\,f(s)}{(x-s)^{\alpha}}, \quad \alpha \in (0,1)
\end{align*}
which is the Riemann-Liouville fractional derivative on the half line $(0, +\infty)$ (\cite{SKM93}). For $f:(-\infty, +\infty) \mapsto \mathbb{R}$, we also obtain that 
\begin{align*}
\partial_x^\alpha f (x) = & \frac{\partial}{\partial x} \left[ \frac{1}{\Gamma(1-\alpha)} \int_0^\infty ds\, s^{-\alpha} f(x-s) \right]\\
= & \frac{\partial}{\partial x} \left[ \frac{1}{\Gamma(1-\alpha)} \int_{-\infty}^x ds\, (x-s)^{-\alpha} f(s) \right]\\
= & \frac{1}{\Gamma(1-\alpha)} \frac{\partial}{\partial x} \int_{-\infty}^x \frac{ds\,f(s)}{(x-s)^{\alpha}}, \quad \alpha \in (0,1)
\end{align*}
which is the Weyl derivative on the whole real line $(-\infty, +\infty)$, see for example \citet{SKM93}. 
\end{os}

\section{Fractional advection equations}
\label{secIII}
This section has been inspired by the works by \citet{MBB99,MeeMorWh06} where the authors studied the advection-dispersion equation 
\begin{equation*}
\frac{\partial \rho}{\partial t} = - \mathbf{v} \cdot \nabla \rho (\mathbf{x}, t)  + c \, \mathbb{D}^\alpha_M \rho(\mathbf{x}, t)
\end{equation*}
$c>0$, introduced to model anomalous dispersion in ground water flow. The fractional operator $\mathbb{D}^\alpha_M$ is defined in the space of Fourier transforms in the sense that $\mathbb{D}^\alpha_M \rho(\mathbf{x}, t)$ is given as inverse of $(-i\boldsymbol{\xi} \cdot \boldsymbol{\theta})^\alpha \widehat{\rho}(\boldsymbol{\xi}, t)$ with $\|\boldsymbol{\theta} \| =1$. Thus, the operator $\mathbb{D}^\alpha_M$ generalizes the directional derivative $\mathbb{D}^1_M = \boldsymbol{\theta}\cdot \nabla $. Here we study the advection equations in which the fractional directional derivative has been previously defined in \eqref{deffracdirder} and find results based on the operational rule \eqref{opRule1}.

We will show that, for the operator $\mathcal{A}=(\mathbf{a}\cdot \nabla)$, the solution to the space fractional equation
\begin{equation}
\left( \frac{\partial }{\partial t} + \mathcal{A}^\alpha \right) u_\alpha = 0 \label{eqAalpha}
\end{equation}
subject to the initial condition 
$$u_\alpha(x, 0)=u_0(x)$$ 
can be represented as the convolution
\begin{equation}
u_\alpha(x, t) = T^\alpha_t \,u_0(x)  \label{convAalpha}
\end{equation}
where
\begin{equation*}
T^\alpha_t u_0(x) = \mathbb{E} e^{- \mathfrak{H}^{\alpha}_t \, \mathcal{A}} \, u_0(x)=  \int_0^\infty d\phi_t(s) \, T_s\, u_0(x).
\end{equation*}
The semigroup $d\phi_t(s)=dsh_{\alpha}(s, t)$ represents the law of the positively skewed L\'evy process $\mathfrak{H}^\alpha_t$, $t>0$ and plays, in this context, the role of $\gamma(s)$ in \eqref{deffracdirder2}. Due to the fact that $\mathfrak{H}^\alpha_t \stackrel{a.s.}{\longrightarrow} t$ (which is the elementary subordinator) for $\alpha \to 1$, we obtain that
$$\lim_{\alpha \to 1}h_\alpha(x, s) = \delta(x-s) $$
and therefore 
$$u(x, t) = T_t\, u_0(x)$$
is the solution to \eqref{eqAalpha} for $\alpha=1$. 

We start our analysis by noticing the following fact. 
\begin{lem}
For the solution to
\begin{equation}
\left( \frac{\partial }{\partial t} + \partial_x^\alpha \right) u_\alpha = 0, \quad x \geq 0, \; t>0 \label{firstderFracp}
\end{equation}
subject to the initial and boundary conditions 
$$u_\alpha (x, 0) = u_0(x), \qquad u_\alpha(0, t)=0,$$
we have that
\begin{equation}
u_\alpha(x, t) = e^{-t \partial_x^\alpha} u_0(x) = \mathbb{E} e^{-\mathfrak{H}^\alpha_t \partial_x}\, u_0(x).
\end{equation}
\end{lem}
\begin{proof}
By considering the function $u: (0, +\infty) \mapsto \mathbb{R}$ with 
\begin{equation}
u\big|_{\partial \mathbb{R}_+} = 0, \label{provided1hold}
\end{equation} 
we get that
\begin{align}
\widehat{\partial_x^\alpha u}(\xi) = & \int_{-\infty}^{+\infty} dx \, e^{i \xi x} \partial_x^\alpha u(x)\notag \\
= & \int_{-\infty}^{+\infty} dx \, e^{i \xi x} \frac{\partial_x}{\Gamma(1-\alpha)} \int_0^\infty dz\, z^{-\alpha} e^{-z \partial_x}u(x) . \notag 
\end{align}
From the Foubini's theorem and by taking into account \eqref{opRule1}, we obtain that
\begin{align}
\widehat{\partial_x^\alpha u}(\xi) = &  \frac{1}{\Gamma(1-\alpha)} \int_0^\infty dz\, z^{-\alpha}  \int_{-\infty}^{+\infty} dx \, e^{i \xi x} \partial_x\, e^{-z \partial_x}u(x) \notag \\
= & \frac{1}{\Gamma(1-\alpha)} \int_0^\infty dz\, z^{-\alpha}  \int_{-\infty}^{+\infty} dx \, e^{i \xi x} \partial_x\, u(x - z)  \notag \\
= & \frac{1}{\Gamma(1-\alpha)} \int_0^\infty dz\, z^{-\alpha}  \int_{z}^{+\infty} dx \, e^{i \xi x} \partial_x\, u(x - z) \notag\\
= & \frac{-i \xi}{\Gamma(1-\alpha)} \int_0^\infty dz\, z^{-\alpha}  \int_{z}^{+\infty} dx \, e^{i \xi x} u(x - z) \notag
\end{align}
provided that \eqref{provided1hold} holds true. Furthermore, 
\begin{align}
\widehat{\partial_x^\alpha u}(\xi) = & \frac{-i\xi}{\Gamma(1-\alpha)} \int_0^\infty dz\,  z^{-\alpha} \int_{z}^{+\infty} dx \, e^{i \xi x} u(x-z)\notag \\
= & \frac{-i\xi}{\Gamma(1-\alpha)} \int_0^\infty dz\, z^{-\alpha} e^{i \xi z} \int_{0}^{+\infty} dx \, e^{i \xi x}u(x) \notag \\
= & \frac{-i\xi}{\Gamma(1-\alpha)} \int_0^\infty dz \, z^{-\alpha} e^{i  z \xi}\, \widehat{u}(\xi) \notag \\
= & (-i\xi)^\alpha \, \widehat{u}(\xi) = |\xi |^\alpha e^{- i \frac{\pi \alpha}{2} \frac{\xi}{|\xi |}} \, \widehat{u}(\xi), \quad \alpha \in (0,1) \label{furAalpha}
\end{align}
from the fact that formula \eqref{FurdistrZ} holds and provided that $u \in L^1(\mathbb{R}_+)$. From \eqref{furAalpha}, the equation 
$$ \frac{\partial u_\alpha}{\partial t} = -\partial_x^\alpha u_\alpha$$
can be rewritten by passing to the Fourier transforms as
$$ \frac{\partial \widehat{u_\alpha}}{\partial t}(\xi, t) = -\widehat{\partial_x^\alpha u_\alpha}(\xi, t) = -|\xi |^\alpha e^{- i \frac{\pi \alpha}{2} \frac{\xi}{|\xi |}} \, \widehat{u_\alpha}(\xi, t)$$ 
which, in turns, leads to
\begin{equation}
\widehat{u_\alpha}(\xi, t) = \widehat{u_0}(\xi) \exp\left( -t |\xi |^\alpha e^{- i \frac{\pi \alpha}{2}\frac{\xi}{|\xi |}} \right). \label{coineq47}
\end{equation}
If $ \mathcal{A}=\partial_x$ in \eqref{convAalpha}, then from the operational rule \eqref{opRule1}, we obtain that 
$$u(x, t) = e^{-t\partial_x }u_0(x) = u_0(x - t)$$ 
and thus
\begin{equation}
u_\alpha(x, t) = e^{-t \partial_x^\alpha} u_0(x) = \int_0^\infty ds\, h_\alpha(s, t) \, u_0(x-s). \label{211eq}
\end{equation}
The Fourier transform of the convolution \eqref{211eq} is therefore written as
\begin{equation}
\widehat{u_\alpha}(\xi, t) = \int_0^\infty ds\, h_\alpha(s, t) \, e^{i \xi s} \widehat{u_0}(\xi) = \widehat{u_0}(\xi) \, \widehat{ h_\alpha}(\xi, t)\label{eq47}
\end{equation}
where $\widehat{h_\alpha}(\xi, t) = \mathbb{E}\exp  i \xi \mathfrak{H}^\alpha_t$ is that in \eqref{furHPrel} and thus \eqref{eq47} coincides with \eqref{coineq47}. From this we obtain the claimed result.
\end{proof}

\begin{os}
\normalfont
If we assume that $\widehat{u_\alpha}(\xi, 0)=1$, that is $u_\alpha(x, 0)=\delta(x)$, than we get that $u_\alpha=h_\alpha$ which is the law of a stable subordinator $\mathfrak{H}^{\alpha}_t$, $t>0$ with $\alpha \in (0,1)$.
\end{os}

\begin{lem}
Let $u : \mathbf{D} \subseteq \mathbb{R}^n_+ \mapsto \mathbb{R}$ be a function such that  $u \in Dom(\mathbf{a}\cdot \nabla)$. For the fractional operator \eqref{deffracdirder} we have that
\begin{equation}
\widehat{(\mathbf{a} \cdot \nabla)^\alpha u}(\boldsymbol{\xi}) = (-i \, \mathbf{a} \cdot \boldsymbol{\xi})^\alpha \, \widehat{u}(\boldsymbol{\xi}). \label{furAnalpha}
\end{equation}
\label{lemmaFurDirDer}
\end{lem}
\begin{proof}
Let us consider the function $u : \mathbf{D} \subseteq \mathbb{R}^n \mapsto \mathbb{R}$ with 
\begin{equation*}
u(\mathbf{x}) \Big|_{\partial \mathbf{D}} = 0.
\end{equation*}
We have that
\begin{align*}
\widehat{(\mathbf{a} \cdot \nabla)^\alpha u}(\boldsymbol{\xi}) = & \int_{\mathbb{R}^n} d\mathbf{x} \,  e^{i\boldsymbol{\xi} \cdot \mathbf{x}} \frac{(\mathbf{a}\cdot \nabla)}{\Gamma(1 - \alpha)} \int_0^\infty dz\,z^{-\alpha} e^{-z (\mathbf{a} \cdot \nabla)}u(\mathbf{x})\\
= & \frac{(-i\mathbf{a}\cdot \boldsymbol{\xi})}{\Gamma(1 - \alpha)} \int_{\mathbf{D}} d\mathbf{x} \,  e^{i\boldsymbol{\xi} \cdot \mathbf{x}}\int_0^\infty dz\,z^{-\alpha} e^{-z (\mathbf{a} \cdot \nabla)}u(\mathbf{x}).
\end{align*}
From the Foubini's theorem we can write
\begin{align*}
\widehat{(\mathbf{a} \cdot \nabla)^\alpha u}(\boldsymbol{\xi}) = & \frac{(-i\mathbf{a}\cdot \boldsymbol{\xi})}{\Gamma(1 - \alpha)} \int_0^\infty dz\,z^{-\alpha} \int_{\mathbf{D}} d\mathbf{x} \,  e^{i\boldsymbol{\xi} \cdot \mathbf{x}} e^{-z (\mathbf{a} \cdot \nabla)} u(\mathbf{x})
\end{align*}
where
\begin{align*}
\int_{\mathbf{D}} d\mathbf{x} \,  e^{i\boldsymbol{\xi} \cdot \mathbf{x}} e^{-z (\mathbf{a} \cdot \nabla)} u(\mathbf{x}) = & \int_{D+ \mathbf{a}z} dx_1 \cdots dx_n \, e^{i\boldsymbol{\xi} \cdot \mathbf{x}} u(x_1-a_1z, \ldots, x_n - a_n z)\notag \\
= & e^{ i z (a_1\xi_1 + \cdots + a_n\xi_n) } \int_{\mathbf{D}} dx_1 \cdots dx_n e^{i\boldsymbol{\xi} \cdot \mathbf{x}}u(x_1, \ldots, x_n) \\
= & e^{i z \,\mathbf{a} \cdot \boldsymbol{\xi}}\, \widehat{u}(\boldsymbol{\xi}).
\end{align*}
From \eqref{FurdistrZ}, by considering that
\begin{equation*}
|\xi |^{-\eta -1} e^{i \frac{\pi}{2}\frac{\xi}{|\xi |} (\eta +1)} \Bigg|_{\xi = \mathbf{a} \cdot \boldsymbol{\xi}} = (-i\xi)^{-\eta -1} \Bigg|_{\xi = \mathbf{a} \cdot \boldsymbol{\xi}}
\end{equation*}
we obtain
\begin{align*}
\widehat{(\mathbf{a} \cdot \nabla)^\alpha u}(\boldsymbol{\xi}) = & \frac{(-i\mathbf{a}\cdot \boldsymbol{\xi})}{\Gamma(1 - \alpha)} \int_0^\infty dz\,z^{-\alpha} e^{i z \,\mathbf{a} \cdot \boldsymbol{\xi}}\, \widehat{u}(\boldsymbol{\xi}) =  (-i \, \mathbf{a} \cdot \boldsymbol{\xi})^\alpha \, \widehat{u}(\boldsymbol{\xi})
\end{align*}
and formula \eqref{furAnalpha} is proved.
\end{proof}

We present the main result of this section.
\begin{te}
We have that:
\begin{itemize}
\item [i)] for $\beta \in (0, 1]$ and $\alpha \in (0,1)$, the solution to the fractional equation
\begin{equation}
\left( \frac{\partial^\beta}{\partial t^\beta} + (\mathbf{a}\cdot \nabla)^\alpha\right) u_{\alpha, \beta}(\mathbf{x}, t) = 0, \quad \mathbf{x} \in \mathbb{R}^n_{+}, \, t>0 \label{eq1MainT}
\end{equation}
subject to the initial and boundary conditions 
$$u_{\alpha, \beta}(\mathbf{x}, 0) = u_0(\mathbf{x}), \quad u_{\alpha, \beta}(\mathbf{0}, t) = 0$$ 
is written as
\begin{equation}
u_{\alpha, \beta}(\mathbf{x}, t) = \mathcal{U}^{\alpha}_{\beta}(\mathbf{a}\cdot \mathbf{x}, t), \label{solab}
\end{equation}
where $\mathcal{U}^{\alpha}_{\beta}$ is the solution to \eqref{eqPrelim1}.
\item [ii)] for $\beta \in (0,1)$, the solution to the fractional equation 
\begin{equation}
\left( \partial^\beta_t + (\mathbf{a} \cdot \nabla) \right) v_\beta(\mathbf{x},t) = 0, \quad \mathbf{x} \in \mathbb{R}^{n}_{+}, \; t>0, \;\beta \in (0,1) \label{eqMAIN}
\end{equation}
subject to the initial and boundary conditions
$$v_\beta(\mathbf{x}, 0)=\delta(\mathbf{x}), \qquad v_\beta(\mathbf{0}, t)=t^{\nu}_{+}, \; \nu > - 1 $$
is given by
\begin{equation}  
v_\beta(\mathbf{x}, t) = t^{\nu} W_{-\beta, \nu +1}\left( - \frac{\mathbf{a} \cdot \mathbf{x} }{t^{\beta}} \right).\label{solab2}
\end{equation}
\item [iii)] for $\nu = -\beta$, the problem \textnormal{ii)} can be rewritten as
\begin{equation}
\left( \frac{\partial^\beta}{\partial t^\beta} + (\mathbf{a} \cdot \nabla) \right) v_\beta(\mathbf{x},t) = 0, \quad \mathbf{x} \in \mathbb{R}^{n}_{+}, \; t>0, \;\beta \in (0,1) \label{eqMAIN2}
\end{equation}
subject to the initial condition
$$v_\beta(\mathbf{x}, 0)=\delta(\mathbf{x})$$
and the solution \eqref{solab2} takes the form
\begin{equation}
v_\beta(\mathbf{a}\cdot \mathbf{x}, t) = l_\beta(\mathbf{a}\cdot \mathbf{x}, t)
\end{equation}
where $l_\beta(x, t)$, $(x,t) \in (0, +\infty)^2$ is the law of $\mathfrak{L}^\beta_t$, $t>0$.
\end{itemize}
\label{mainTheo}
\end{te}
\begin{proof}
Proof of i). From the previous result the Fourier transform of the equation \eqref{eq1MainT} is written as
\begin{equation*}
\frac{\partial^\beta \widehat{u_{\alpha, \beta}}}{\partial t^\beta}(\boldsymbol{\xi}, t) = - (-i \mathbf{a}\cdot \boldsymbol{\xi})^\alpha \widehat{u_{\alpha, \beta}}(\boldsymbol{\xi}, t).
\end{equation*}
The Mittag-Leffler function is an eigenfunction of the Dzerbayshan-Caputo fractional derivative and therefore we get that
\begin{equation}
\widehat{u_{\alpha, \beta}}(\boldsymbol{\xi}, t) = \widehat{u_0}(\boldsymbol{\xi})\, E_{\beta}\left(-t^{\beta}(-i \mathbf{a}\cdot \boldsymbol{\xi})^\alpha \right). \label{vcoinu}
\end{equation}
This is equivalent to write $u_{\alpha, \beta}$ as follows
\begin{align*}
u_{\alpha, \beta}(\mathbf{x}, t) = & \int_0^\infty ds\, l_{\beta}(s, t) e^{-s (\mathbf{a}\cdot \nabla)^\alpha} u_0(\mathbf{x}) \\
= & \int_{0}^{\infty}ds\, \mathcal{U}^{\alpha}_{\beta}(s, t) e^{-s \, \mathbf{a}\cdot \nabla} u_0(\mathbf{x}), \quad \alpha, \beta \in (0,1)
\end{align*}
where, we recall that,
\begin{equation}
\mathcal{U}^{\alpha}_{\beta}(x, t) = \int_0^\infty h_\alpha (x, s) l_\beta(s, t)ds  \label{TMPUlapq}
\end{equation}
is that in \eqref{funLamp}. From \eqref{opRule1} we get that
\begin{align*}
u_{\alpha, \beta}(\mathbf{x}, t) = &  \int_{0}^{\infty}ds\, \mathcal{U}^{\alpha}_{\beta}(s, t) e^{-s \, \mathbf{a}\cdot \nabla} u_0(\mathbf{x})\\
= & \int_{0}^{\infty}ds\, \mathcal{U}^{\alpha}_{\beta}(s, t) \, u_0(x_1 - sa_1, \ldots, a_n - a_ns_n)
\end{align*}
and thus
\begin{align*}
\widehat{u_{\alpha, \beta}}(\boldsymbol{\xi}, t) = & \int_{0}^{\infty}ds\, \mathcal{U}^{\alpha}_{\beta}(s, t) \, e^{i \,s\, \mathbf{a}\cdot \boldsymbol{\xi}} \, \widehat{u_0}(\xi_1, \ldots , \xi_n)\\
= & \widehat{u_0}(\boldsymbol{\xi})\, \widehat{\mathcal{U}^{\alpha}_{\beta}}(\mathbf{a}\cdot \boldsymbol{\xi}, t).
\end{align*}
From  \eqref{lapHspace}, \eqref{lapLspace} and \eqref{TMPUlapq} we have that
\begin{align*}
\widehat{\mathcal{U}^{\alpha}_{\beta}}(\gamma, t) = & \int_0^\infty e^{i \gamma x}\mathcal{U}^{\alpha}_{\beta}(x, t)dx\\
= & \int_0^\infty e^{-s (-i \gamma)^\alpha}l_{\beta}(s, t)ds\\
= & E_{\beta}(-t^\beta (-i \gamma)^\alpha), \quad \gamma \in \mathbb{R}
\end{align*}
and thus
\begin{align*}
\widehat{u_{\alpha, \beta}}(\boldsymbol{\xi}, t) = & \widehat{u_0}(\boldsymbol{\xi})\,\widehat{\mathcal{U}^{\alpha}_{\beta}}(\gamma, t) \Big|_{\gamma = \mathbf{a}\cdot \boldsymbol{\xi}} = \widehat{u_0}(\boldsymbol{\xi})\, E_{\beta}\left(-t^\beta (-i \mathbf{a} \cdot \boldsymbol{\xi})^\alpha \right).
\end{align*}

Proof of ii). The Laplace transform of $v_\beta$ is written as
\begin{align}
\widetilde{v_\beta}(\mathbf{x}, \lambda) = & \int_0^\infty e^{-\lambda t}  v_{\beta}(\mathbf{x}, t)\, dt\nonumber \\
= & \int_0^\infty e^{-\lambda t} t^{\nu} W_{-\beta, \nu +1}\left( - \frac{\mathbf{a} \cdot \mathbf{x} }{t^{\beta}} \right)\, dt\nonumber \\
= & \int_0^\infty e^{-\lambda t} t^{\nu} \sum_{k=0}^\infty \left( - \frac{\mathbf{a} \cdot \mathbf{x} }{t^{\beta}} \right)^k \frac{dt}{k!\, \Gamma(-\beta k + \nu + 1)}\nonumber \\
= & \lambda^{-\nu -1} \sum_{k=0}^\infty \left( - \lambda^\beta \mathbf{a} \cdot \mathbf{x}  \right)^k \frac{1}{k!}\nonumber \\
= &  \lambda^{-\nu -1} \exp\left( - \lambda^\beta \mathbf{a} \cdot \mathbf{x}  \right) \label{lapx}
\end{align}
where, from \eqref{LapdistrZ},
\begin{equation*}
\lambda^{-\nu-1} = \int_0^\infty e^{-\lambda t} v_\beta(\mathbf{0},t )dt= \widetilde{v_\beta}(\mathbf{0}, \lambda) .
\end{equation*}
We immediately get that
$$\frac{\partial}{\partial x_j} \widetilde{v_\beta}(\mathbf{x}, \lambda) = -\lambda^\beta \, a_j\, \widetilde{v_\beta}(\mathbf{x}, \lambda)$$
and thus
\begin{equation*}
\sum_{j=1}^n a_j \frac{\partial}{\partial x_j} \widetilde{v_\beta}(\mathbf{x}, \lambda) = -\lambda^\beta \, \sum_{j=1}^n a_j^2\, \widetilde{v_\beta}(\mathbf{x}, \lambda) = -\lambda^\beta \, \widetilde{v_\beta}(\mathbf{x}, \lambda)
\end{equation*}
being $\|\mathbf{a}\| = 1$. This means that
\begin{equation*}
(\mathbf{a}\cdot \nabla) \widetilde{v_\beta}(\mathbf{x}, \lambda)  = - \lambda^\beta \, \widetilde{v_\beta}(\mathbf{x}, \lambda) = \int_0^\infty e^{-t \lambda} \left(-\partial^\beta_t \,v_\beta(\mathbf{x}, t) \right)\, dt
\end{equation*}
where in the last identity we have considered the Laplace transform of the Riemann-Liouville derivative (\cite{SKM93}). 

Proof of iii). The Fourier transform of \eqref{eqMAIN2} is written as
\begin{equation*}
\frac{\partial^\beta \widehat{v_{\beta}}}{\partial t^\beta}(\boldsymbol{\xi}, t) = (i \mathbf{a}\cdot \boldsymbol{\xi})\, \widehat{v_{\beta}}(\boldsymbol{\xi}, t)
\end{equation*}
and leads to
\begin{equation*}
\widehat{v_{\beta}}(\boldsymbol{\xi}, t) = E_\beta\left( i t^\beta \, \mathbf{a}\cdot \boldsymbol{\xi} \right)
\end{equation*}
which coincides with \eqref{vcoinu} for $\alpha=1$. We recall that
\begin{equation*}
\int_0^\infty e^{-\gamma x} \, t^{-\beta} W_{-\beta, 1-\beta}\left( -\frac{x}{t^{\beta}} \right)\, dx =  \frac{1}{\gamma t^{\beta}} E_{-\beta, 1-\beta}\left( - \frac{1}{\gamma t^{\beta}} \right), \quad \gamma >0
\end{equation*}
where (see \citet{OB09EJP})
\begin{equation*}
 \frac{1}{\gamma t^{\beta}} E_{-\beta, 1-\beta}\left( - \frac{1}{\gamma t^{\beta}} \right) = E_{\beta, 1}(- \gamma t^{\beta})
\end{equation*}
and, from \eqref{opRule1},
\begin{align*}
\int_{\mathbb{R}^n_{+}} d\mathbf{x}\, e^{i \boldsymbol{\xi}\cdot \mathbf{x}} e^{-s (\mathbf{a}\cdot \nabla)} v_{\beta}(\mathbf{x}, 0) = & \int_{\mathbb{R}^n_{+}} d\mathbf{x}\, e^{i \boldsymbol{\xi}\cdot \mathbf{x}} v_{\beta}(x_1-sa_1, \ldots , x_n - sa_n,  0)\\
= & e^{i \, s\,  \mathbf{a} \cdot \boldsymbol{\xi}} \int_{\mathbb{R}^n_{+}} d\mathbf{x}\, e^{i \boldsymbol{\xi}\cdot \mathbf{x}} v_{\beta}(x_1, \ldots , x_n,  0)\\
= & e^{i \, s\,  \mathbf{a} \cdot \boldsymbol{\xi}} \, \widehat{v_\beta}(\boldsymbol{\xi})
\end{align*}
where $ \widehat{v_\beta}(\boldsymbol{\xi})=1$. Thus, we obtain that
\begin{align*}
E_\beta\left( i t^\beta \, \mathbf{a}\cdot \boldsymbol{\xi} \right) = & \int_0^\infty e^{i \, s\, \mathbf{a}\cdot \boldsymbol{\xi}} \,  t^{-\beta} W_{-\beta, 1-\beta}\left( -\frac{s}{t^{\beta}} \right)\, ds\\
= & \int_0^\infty ds \left( \int_{\mathbb{R}^n_{+}} d\mathbf{x}\, e^{i \boldsymbol{\xi}\cdot \mathbf{x}} e^{-s (\mathbf{a}\cdot \nabla)} v_{\beta}(\mathbf{x}, 0)  \right)  t^{-\beta} W_{-\beta, 1-\beta}\left( -\frac{s}{t^{\beta}} \right)\\
= & \int_{\mathbb{R}^n_{+}} e^{i \boldsymbol{\xi}\cdot \mathbf{x}} \left(  \int_0^\infty ds \, t^{-\beta} W_{-\beta, 1-\beta}\left( -\frac{s}{t^{\beta}} \right)\, e^{-s (\mathbf{a}\cdot \nabla)} v_{\beta}(\mathbf{x}, 0)  \right)  d\mathbf{x} \\
= & \left[\textrm{from the operational rule } \eqref{opRule1}\right]\\
= & \int_{\mathbb{R}^n_{+}} e^{i \boldsymbol{\xi}\cdot \mathbf{x}} \left(  \int_0^\infty ds \, t^{-\beta} W_{-\beta, 1-\beta}\left( -\frac{s}{t^{\beta}} \right)\, v_{\beta}(\mathbf{x} - s\mathbf{a}, 0)  \right)  d\mathbf{x}
\end{align*}
where $v_\beta(\mathbf{x}, 0)=\delta(\mathbf{x})$ and $\mathbf{x} = s\mathbf{a}$ if and only if $\mathbf{x} \cdot \mathbf{a}= s \|\mathbf{a} \|^2$ with $\|\mathbf{a} \|^2=1$. We obtain that
\begin{align*}
E_\beta\left( i t^\beta \, \mathbf{a}\cdot \boldsymbol{\xi} \right) = & \int_{\mathbb{R}^n_{+}} e^{i \boldsymbol{\xi}\cdot \mathbf{x}} \left(  \int_0^\infty ds \, t^{-\beta} W_{-\beta, 1-\beta}\left( -\frac{s}{t^{\beta}} \right)\, \delta(\mathbf{a} \cdot \mathbf{x} - s)  \right)  d\mathbf{x}\\
= & \int_{\mathbb{R}^n_{+}} e^{i \boldsymbol{\xi}\cdot \mathbf{x}} v_{\beta}(\mathbf{x}, t)  \, d\mathbf{x}
\end{align*}
which concludes the proof.
\end{proof}

\begin{os}
\normalfont
We remark that, for $\alpha=\beta$, formula \eqref{solab} becomes
\begin{equation*}
u_{\beta, \beta}(\mathbf{x}, t) = \frac{\sin \beta \pi}{\pi} \frac{(\mathbf{a} \cdot \mathbf{x})^{\beta -1}\, t^{\beta} }{(\mathbf{a} \cdot \mathbf{x})^{2\beta} + 2 (\mathbf{a} \cdot \mathbf{x})^{\beta} t^\beta \cos \beta \pi + t^{2\beta}}, \quad \beta \neq 1.
\end{equation*}
The law $u_{\beta, \beta}(x, t)$, $x \geq 0$, $t>0$ can be interpreted as the density law of the ratio of two independent stable subordinators both of order $\beta \in (0,1)$, see \citet{Dov2,Lanc}. This ratio (formula \eqref{ratioSS}) does not depend on $t>0$. \label{Remaratio}
\end{os}

\begin{os}
\normalfont
For $\alpha=1$ and $\beta \in (0,1)$, the solution \eqref{solab2} (and \eqref{solab}) with $v_\beta(\mathbf{0}, t)=t^{-\beta}$, can be written as
\begin{equation*}
v_\beta(\mathbf{x}, t) = t^{-\beta}W_{-\beta, 1-\beta}\left( - \frac{\mathbf{a}\cdot \mathbf{x}}{t^\beta} \right) = l_\beta(\mathbf{a}\cdot \mathbf{x}, t)
\end{equation*}
where $l_\beta$ is the law of $\mathfrak{L}^\beta_t$, $t>0$.
\end{os}

\begin{os}
\label{RemfurRemH}
\normalfont
For $\beta=1$ and $\alpha \in (0,1)$, the solution \eqref{solab} takes the form
\begin{equation}
u_{\alpha}(\mathbf{x}, t) =u_{\alpha, 1}(\mathbf{x}, t) = h_{\alpha}(\mathbf{a} \cdot \mathbf{x}, t) \label{solHRemark}
\end{equation}
where $h_\alpha$ is the density law of $\mathfrak{H}^\alpha_t$ and solves
\begin{equation}
\left( \frac{\partial }{\partial t} + ( \mathbf{a}\cdot \nabla )^\alpha \right) u_\alpha(\mathbf{x}, t) = 0, \quad (\mathbf{x},t) \in \mathbb{R}^n_{+}\times (0, +\infty) \label{probAalpha}
\end{equation}
subject to the initial condition 
\begin{equation}
u_\alpha (\mathbf{x}, 0) = u_0(\mathbf{x})=\prod_{k=1}^n u_0(x_k) \label{condiniProd}
\end{equation}
with $u_0(x_k) = \delta(x_k)$ for all $k=1,2, \ldots , n$. Indeed, we have that
\begin{equation}
u_\alpha(\mathbf{x}, t) = e^{-t ( \mathbf{a}\cdot \nabla )^\alpha} u_0(\mathbf{x}) = \int_0^\infty ds\, h_\alpha(s, t) \, e^{-s ( \mathbf{a}\cdot \nabla )}u_0(\mathbf{x}) \label{eq417K}
\end{equation}
and
\begin{equation}
\widehat{u_\alpha}(\boldsymbol{\xi}, t) = \widehat{u_0}(\boldsymbol{\xi})\, \widehat{h_\alpha}(\mathbf{a}\cdot \boldsymbol{\xi}, t) \label{furRemH}
\end{equation}
where
\begin{equation}
\widehat{u_0}(\boldsymbol{\xi}) = \prod_{k=1}^n \widehat{u_0}(\xi_k) \label{FurcondiniProd}
\end{equation}
From \eqref{opRule1} and the fact that $u_0(x_k) = \delta(x_k)$ for all $k$, we rewrite \eqref{eq417K} as follows
\begin{equation*}
u_\alpha(\mathbf{x}, t) = \int_0^\infty ds\, h_\alpha(s, t) \, \delta(\mathbf{x}-s\mathbf{a}). 
\end{equation*}
By considering that $\mathbf{x}=s\mathbf{a}$ if and only if $\mathbf{a} \cdot \mathbf{x}=s$ we obtain \eqref{solHRemark}. From \eqref{opRule1} and \eqref{condiniProd} we have that
\begin{equation*}
e^{-s ( \mathbf{a}\cdot \nabla )}u_0(\mathbf{x}) = \prod_{k=1}^n u_0(x_k - s a_k)
\end{equation*}
and therefore we obtain the Fourier transform
\begin{align*}
\widehat{u_\alpha}(\boldsymbol{\xi}, t) =  & \int_{\mathbb{R}^n_{+}} e^{i \boldsymbol{\xi}\cdot \mathbf{x}} \int_0^\infty ds\, h_\alpha(s, t) \, \prod_{i=k}^n u_0(x_k - s a_k)\, d\mathbf{x}\\
= & \prod_{k=1}^n \widehat{u_0}(\xi_k) \, \int_0^\infty ds\, h_\alpha(s, t) e^{i s\, \mathbf{a} \cdot \boldsymbol{\xi}}\\
= & \widehat{u_0}(\boldsymbol{\xi}) \, \widehat{h_\alpha}(\mathbf{a} \cdot \boldsymbol{\xi}, t).
\end{align*}
Formula \eqref{furHPrel} says that
\begin{equation*}
\widehat{h_\alpha}(\mathbf{a} \cdot \boldsymbol{\xi}, t )= \exp\left( -t |\mathbf{a} \cdot \boldsymbol{\xi} |^\alpha e^{- i \frac{\pi \alpha}{2}\frac{\mathbf{a} \cdot \boldsymbol{\xi}}{| \mathbf{a} \cdot \boldsymbol{\xi} |}} \right) = \exp \left( - t (-i \mathbf{a} \cdot \boldsymbol{\xi})^\alpha \right).
\end{equation*}
By taking into consideration \eqref{furAnalpha}, from the fact that
$$\frac{\partial \widehat{h_\alpha}}{\partial t} (\mathbf{a} \cdot \boldsymbol{\xi}, t) = - (- i \, \mathbf{a} \cdot \boldsymbol{\xi}) \, \widehat{h_\alpha}(\mathbf{a} \cdot \boldsymbol{\xi}, t)$$
 we obtain \eqref{probAalpha}.
\end{os}

From the previous result we arrive at the following statement.
\begin{coro}
For $\beta \in (0,1)$, the solution to the time-fractional equation
\begin{equation}
\Big( \partial^\beta_t + (\mathbf{a}\cdot \nabla) \Big)v_\beta (\mathbf{x}, t)= 0, \quad (\mathbf{x},t) \in \mathbb{R}^n_+\times (0, +\infty)
\end{equation}
subject to the initial and boundary conditions
\begin{equation*}
v_\beta(\mathbf{x}, 0)=\delta(\mathbf{x}), \qquad v_\beta(\mathbf{0}, t)=t^{-n\beta}_{+}
\end{equation*}
or equivalently
\begin{equation}
\Big( \frac{\partial^{n\beta}}{\partial t^{n\beta}} + (\mathbf{a}\cdot \nabla) \Big)v_\beta (\mathbf{x}, t)= 0, \quad (\mathbf{x},t) \in \mathbb{R}^n_+\times (0, +\infty)
\end{equation}
subject to the initial conditions
\begin{equation*}
v_\beta(\mathbf{x}, 0)=\delta(\mathbf{x}), \quad \frac{\partial^k v_\beta}{\partial t^k}(\mathbf{x}, 0) = 0, \; 0<k< \lceil n\beta \rceil -1
\end{equation*}
($\lceil \cdot \rceil$ is the smallest following integer) is given by 
\begin{equation}
v_\beta(\mathbf{x}, t) = \frac{1}{t^{n\beta}} W_{-\beta, 1-n \beta}\left(- \frac{\mathbf{a} \cdot \mathbf{x}}{t^{\beta}} \right). 
\label{lawvW}
\end{equation}
\end{coro}
\begin{proof}
It suffices to consider $\nu=-n\beta$ in Theorem \ref{mainTheo} and the fact that
\begin{equation*}
\widetilde{\frac{\partial^{n\beta} u}{\partial t^{n\beta}}}(\lambda) = \lambda^{n\beta} \widetilde{u}(\lambda) - \sum_{k=0}^{\lceil n\beta \rceil -1} \lambda^{n\beta - k - 1} \frac{\partial^k u}{\partial t^k}(0^+). 
\end{equation*}
\end{proof}

\begin{os}
\label{RemarkX}
\normalfont
We point out that the normalization of \eqref{lawvW} given by
\begin{equation}
p_{\beta}(\mathbf{x}, t; n) = \frac{a_{(n)}}{t^{n\beta}} W_{-\beta, 1-n \beta}\left(- \frac{\mathbf{a} \cdot \mathbf{x}}{t^{\beta}} \right), \quad \mathbf{x} \in \mathbb{R}^n_{+},\, t>0 \label{distrINVW}
\end{equation}
where $a_{(n)} = a_1 \cdot a_2 \cdots a_n$, is a probability distribution whose one-dimensional marginals coincide with $l_\beta$, that is the density law of the inverse process $\mathfrak{L}^\beta_t$, $t>0$. Indeed, for the law $p_{\beta}(\mathbf{x}, \mathbf{t}; n)$, $\beta \in (0,1)$, $(\mathbf{x}, \mathbf{t}) \in \mathbb{R}^n_+ \times \mathbb{R}^n_+$, we can show that
\begin{equation}
\int_{\mathbb{R}^k_+} p_{\beta}(\mathbf{x}, \mathbf{t}; n) d\mathbf{x} = p_{\beta}(\mathbf{y}, \mathbf{t}; n-k), \quad \mathbf{y} \in \mathbb{R}^{n-k}, \quad \mathbf{t} \in \mathbb{R}^n_+ \label{margInt}
\end{equation}
and also that
\begin{equation}
\int_{\mathbb{R}^n_+} p_{\beta}(\mathbf{x}, \mathbf{t}; n) d\mathbf{x} = 1. \label{unityInt}
\end{equation}
To this end, we recall that (see for example \citet{LE})
$$\frac{1}{\Gamma(-\alpha)} = \frac{1}{2\pi i} \int_{\mathcal{C}} ds\, e^{s} s^{\alpha}, \quad \alpha \neq 1,2, \ldots .$$
and we write
\begin{align*}
p_{\beta}(\mathbf{x}, \mathbf{t}; n) = & \frac{a_{(n)}}{(\prod_{j=1}^{n} t_j^\beta)} W_{-\beta, 1-n\beta}\left( - \sum_{j=1}^{n} \frac{a_jx_j}{t_j^\beta} \right)\\
= & \frac{a_{(n)}}{(\prod_{j=1}^{n} t_j^\beta)} \sum_{k=1}^{\infty} \frac{1}{k!} \left( - \sum_{j=1}^{n} \frac{a_jx_j}{t_j^\beta} \right)^k \frac{1}{\Gamma(-\beta k + 1 - n \beta)}\\
= & \frac{a_{(n)}}{(\prod_{j=1}^{n} t_j^\beta)} \frac{1}{2\pi i} \int_{\mathcal{C}}e^s s^{n\beta -1} \sum_{k \geq 0} \frac{1}{k!} \left(- \sum_{j=1}^{n} \frac{a_jx_j}{t_j^\beta} \right)^k s^{\beta k} ds\\
= & \frac{a_{(n)}}{(\prod_{j=1}^{n} t_j^\beta)} \frac{1}{2\pi i} \int_{\mathcal{C}}e^s s^{n\beta -1} \exp\left(- s^\beta \sum_{j=1}^{n} \frac{a_jx_j}{t_j^\beta} \right) ds .
\end{align*}
By observing that
\begin{align*}
\int_0^\infty \exp\left(- s^\beta \sum_{j=1}^{n} \frac{a_j x_j}{t_j^\beta} \right) dx_j = \frac{t^\beta_j}{a_j s^\beta} \exp\left(- s^\beta \sum_{\begin{subarray}{c} j=1\\j \neq j \end{subarray}}^{n} \frac{a_jx_j}{t_j^\beta} \right)
\end{align*}
we get that
\begin{align*}
\int_0^\infty p_{\beta}(\mathbf{x}, \mathbf{t}; n) \, dx_j = & \frac{a_{(n)}}{(\prod_{\begin{subarray}{c} j=1\\j \neq l \end{subarray}}^{n} t_j^{\beta})}\frac{1}{a_l} \frac{1}{2\pi i} \int_{\mathcal{C}}e^s s^{(n-1)\beta -1} \exp\left(- s^\beta \sum_{\begin{subarray}{c} j=1\\j \neq l \end{subarray}}^{n} \frac{a_jx_j}{t_j^\beta} \right) ds\\
= & \frac{a_{(n)}}{(\prod_{\begin{subarray}{c} j=1\\j \neq l \end{subarray}}^{n} t_j^{\beta})} \frac{1}{a_l}  \sum_{k \geq 0} \frac{1}{k!} \left(-  \sum_{\begin{subarray}{c} j=1\\j \neq l \end{subarray}}^{n} \frac{a_jx_j}{t_j^\beta} \right)^k \frac{1}{2\pi i} \int_{\mathcal{C}}e^s s^{(n-1)\beta -1 - \beta k} \,ds\\
= & \frac{a_{(n)}}{(\prod_{\begin{subarray}{c} j=1\\j \neq l \end{subarray}}^{n} t_j^{\beta})}  \frac{1}{a_l} \sum_{k \geq 0} \frac{1}{k!} \left(-  \sum_{\begin{subarray}{c} j=1\\j \neq l \end{subarray}}^{n} \frac{a_jx_j}{t_j^\beta} \right)^k \frac{1}{\Gamma(-\beta k + 1 - (n-1)\beta)}\\
= & \frac{a_{(n)}}{a_l \,(\prod_{\begin{subarray}{c} j=1\\j \neq l \end{subarray}}^{n} t_j^{\beta})} W_{-\beta , 1- (n-1)\beta} \left(-  \sum_{\begin{subarray}{c} j=1\\j \neq l \end{subarray}}^{n} \frac{a_jx_j}{t_j^\beta} \right)
\end{align*}
which is the marginal law of
\begin{equation*}
p_{\beta}(\mathbf{x}, \mathbf{t}; n), \quad \mathbf{x} \in \mathbb{R}^{n}_{+},\; \mathbf{t} \in \mathbb{R}^{n}_{+}
\end{equation*}
and this proves \eqref{margInt}. In order to obtain \eqref{unityInt} it suffices to consider \eqref{margInt} and the fact that $l_\beta$ integrates to unity. Indeed, we recall that
$$p_{\beta}(x, t; 1) = l_\beta(x, t) = \frac{1}{t^{\beta}} W_{-\beta, 1- \beta}\left( - \frac{x}{t^\beta} \right)$$
is the probability density on $[0, +\infty)$ of the process $\mathfrak{L}^\beta_t$, $t>0$.
\end{os}

\begin{os}
\normalfont
The distribution \eqref{distrINVW} can be regarded as the law of a process which is the inverse to the sum of stable subordinators. In particular, let us consider the stable random sheet
$$\left({_1\mathfrak{H}^\beta_{x_1}}, \ldots , {_n\mathfrak{H}^\beta_{x_n}} \right), \quad \mathbf{x} = (x_1, \ldots , x_n) \in \mathbb{R}^n_+$$
with independent stable subordinators ${_j\mathfrak{H}^\beta_{x_j}}$, $j=1,2, \ldots , n$, of order $\beta \in (0,1)$. We are interested in studying the inverse multi parameter process
$$\mathscr{X}_\beta(t)= \left( {_1X_{t}^\beta}, \ldots, {_nX_{t}^\beta}\right), \quad t>0$$ 
of the linear combination
\begin{align*}
\mathscr{H}_\beta(\mathbf{x})  & = \sqrt[\beta]{a_1}\times {_1 \mathfrak{H}^\beta_{x_1}}+ \ldots + \sqrt[\beta]{a_n}\times {_n \mathfrak{H}^\beta_{x_n}}\\
& \stackrel{law}{=} {_1 \mathfrak{H}^\beta_{a_1 x_1}} + \ldots + {_n \mathfrak{H}^\beta_{a_n x_n}}
\end{align*}
with $a_j\geq 0$, $\forall j$, in the sense that
\begin{equation}
Pr\{\mathscr{X}_\beta(t) > \mathbf{x}\} = P\{ {_1X_{t}^\beta} > x_1, \ldots {_nX_{t}^\beta} > x_n \} = P\{\mathscr{H}_\beta(\mathbf{x}) < t  \}. \label{inverseXH}
\end{equation}
We first obtain that
\begin{equation}
Pr\{ \mathscr{H}_\beta(\mathbf{x}) \in ds \}/ds = h_\beta(s, \mathbf{a}\cdot \mathbf{x}) \label{invlaXH}
\end{equation}
Indeed, from \eqref{lapHspace}, we get that
\begin{equation}
\mathbb{E} \exp\left( -\xi \mathscr{H}_\beta(\mathbf{x})  \right) = \exp\left( - \xi^\beta (\mathbf{a}\cdot \mathbf{x}) \right) = \mathbb{E}\exp\left( -\xi \mathfrak{H}^\beta_{(\mathbf{a}\cdot \mathbf{x})}\right)
\end{equation}
which means that
\begin{equation*}
\mathscr{H}_\beta(\mathbf{x}) \stackrel{law}{=} \mathfrak{H}^\beta_{\mathbf{a}\cdot \mathbf{x}}, \quad \mathbf{a}\in \mathbb{R}^n_+,\; \mathbf{x}\in \mathbb{R}^n_+.
\end{equation*}
From \eqref{inverseXH} and \eqref{invlaXH}, we find the law of $\mathscr{X}_\beta$ as follows
\begin{align*}
p_\beta(\mathbf{x}, t) = & \frac{(-\partial)^n}{\partial x_1\cdots \partial x_n} P\{ {_1X_{t}^\beta} > x_1, \ldots {_nX_{t}^\beta} > x_n \}\\
= & \frac{(-\partial)^n}{\partial x_1\cdots \partial x_n} P\{\mathscr{H}_\beta(\mathbf{x}) < t  \}\\
= & \int_0^{t} \frac{(-\partial)^n}{\partial x_1\cdots \partial x_n} P\{\mathscr{H}(\mathbf{x}) \in ds  \}\\
= & \int_0^{t} \frac{(-\partial)^n}{\partial x_1\cdots \partial x_n} \,  h_\beta(s, \mathbf{a}\cdot \mathbf{x})\, ds
\end{align*}
From \eqref{lapHspace} and the fact that
\begin{equation*}
\int_0^\infty e^{-\lambda t}\left( \int_0^{t} h_\beta(s, \mathbf{a}\cdot \mathbf{x})\, ds\right) dt = \frac{1}{\lambda} \int_0^\infty e^{-\lambda s}h_\beta(s, \mathbf{a}\cdot \mathbf{x})\, ds = \frac{1}{\lambda} e^{- (\mathbf{a}\cdot \mathbf{x}) \lambda^\beta}
\end{equation*}
we arrive at the Laplace transform 
\begin{align*}
\int_0^\infty e^{-\lambda t} p_\beta(\mathbf{x}, t) dt = & \int_0^\infty e^{-\lambda t} \left( \int_0^{t} \frac{(-\partial)^n}{\partial x_1\cdots \partial x_n} h_\beta(s , \mathbf{a}\cdot \mathbf{x})\, ds \right) dt\\
= & \frac{(-\partial)^n}{\partial x_1\cdots \partial x_n} \frac{1}{\lambda} e^{- (\mathbf{a}\cdot \mathbf{x}) \lambda^\beta}\\
= & \prod_{j=1}^n a_j \, \lambda^{n\beta -1} e^{- (\mathbf{a}\cdot \mathbf{x}) \lambda^\beta}.
\end{align*}
By recalling that $a_{(n)} = a_1\cdots a_n$ and inverting the above Laplace transform, we get that
\begin{equation*}
p_\beta(\mathbf{x}, t) = \frac{a_{(n)}}{t^{n\beta}} W_{-\beta, 1-n \beta}\left(- \frac{\mathbf{a} \cdot \mathbf{x}}{t^{\beta}} \right)
\end{equation*}
which is the claimed result.
\end{os}

We now extend the results presented in Theorem \ref{mainTheo}. Let us consider the solution to the time-fractional problem 
\begin{equation}
\Big( \partial^\beta_t + (\mathbf{a}\cdot \nabla) \Big) \mathfrak{U}_\beta^{n}(\mathbf{x}, t) = 0, \quad \mathbf{x} \in \mathbb{R}^n_{+},\; t>0,\; \beta \in (0,1) 
\end{equation}
subject to the initial and boundary conditions
$$\mathfrak{U}_\beta^{n}(\mathbf{x}, 0)=\delta(\mathbf{x}), \qquad \mathfrak{U}_\beta^{n}(\mathbf{0}, t)=t^{n - n\beta -1}_{+}, \quad n \in \mathbb{N}.$$
From Theorem \ref{mainTheo} we know that
\begin{equation}
\mathfrak{U}_\beta^{n}(\mathbf{x}, t) = t^{n-n\beta -1}W_{-\beta, n-n\beta}\left( -\frac{\mathbf{a}\cdot \mathbf{x}}{t^{\beta}} \right). \label{lawmultiL0}
\end{equation}

\begin{te}
For $n,m \in \mathbb{N}$, we have that
\begin{equation}
\mathfrak{U}_\beta^{n}(\mathbf{x}, t) * \mathfrak{U}_\beta^{m}(\mathbf{y}, t) = \mathfrak{U}_\beta^{n+m}(\mathbf{x} + \mathbf{y}, t)
\end{equation}
where $*$ stands for the Laplace convolution over $t$. Furthermore, formula \eqref{lawmultiL0} can be written as
\begin{equation}
\mathfrak{U}_\beta^{n}(\mathbf{x}, t) = (u_1 * \cdots *u_n)(\mathbf{x}, t) \label{lawmultiL}
\end{equation}
where
\begin{equation}
u_i(x_i, t) = \frac{1}{t^{\beta}}W_{-\beta, 1-\beta}\left(- \frac{a_i x_i}{t^{\beta}} \right), \quad x\geq 0, t\geq 0,\quad i=1, \ldots, n.
\end{equation}
\end{te} 
\begin{proof}
We first assume that formula \eqref{lawmultiL} holds true and we write
$$\phi(\mathbf{x}, t) = (u_1 * \cdots *u_n)(\mathbf{x}, t).$$
We have that
\begin{align*} 
\widetilde{\phi}(\boldsymbol{\xi}, t) = & \int_{0}^\infty \cdots \int_0^\infty e^{- \boldsymbol{\xi} \cdot \mathbf{x}} \phi(\mathbf{x}, t)\, d\xi_1 \cdots d\xi_n \\
= & \int_0^{\infty} \cdots \int_{0}^{\infty} e^{-\sum_{i=1}^n \xi_i x_i}  (u_1*\cdots *u_n)(x_1, \ldots , x_n, t)\, dx_1 \cdots dx_n\\
= & (\widetilde{u_1}*\cdots *\widetilde{u_2})(\xi_1 , \ldots , \xi_n , t)
\end{align*}
where, from \eqref{wright} and \eqref{mittag-leffler}, 
\begin{equation}
\widetilde{u_i}(\xi_i, t) = \int_0^\infty e^{-\xi_i x_i} u_i(x_i, t) dx_i = \frac{1}{a_i} E_{\beta}\left(- \frac{\xi_i t^{\beta}}{a_i}\right)
\end{equation}
and thus
\begin{align}
\widetilde{\widetilde{\phi}}(\boldsymbol{\xi}, \lambda) = & \int_0^\infty e^{-\lambda t}\, \widetilde{\phi}(\boldsymbol{\xi}, t)\, dt = \prod_{i=1}^{n} \frac{\lambda^{\beta -1}}{(\lambda^{\beta}a_i + \xi_i )}. \label{coinWnu}
\end{align}
Now we consider the function
$$\varphi(\mathbf{x}, t) = t^{n-n\beta -1}W_{-\beta, n-n\beta}\left( \frac{\mathbf{a}\cdot \mathbf{x}}{t^{\beta}} \right).$$
From  \eqref{wright} we have that
\begin{equation}
\widetilde{\varphi}(\mathbf{x}, \lambda) = \int_0^\infty e^{-\lambda t}\, \varphi(\mathbf{x}, t)\,dt = \lambda^{n\beta - n} e^{- \lambda^{\beta} (\mathbf{a}\cdot \mathbf{x})}.
\end{equation}
As we can immediately check
\begin{equation}
\int_{\mathbb{R}^n_+} e^{-\boldsymbol{\xi}\cdot \mathbf{x}} \widetilde{\varphi}(\mathbf{x}, \lambda)\, d\mathbf{x} = \frac{\lambda^{n\beta -n}}{\prod_{i=1}^n (\xi_i + \lambda^\beta a_i)} = \prod_{i=1}^n \frac{\lambda^{\beta -1}}{(\xi_i + \lambda^\beta a_i)}
\end{equation}
coincides with \eqref{coinWnu} and therefore $\phi=\varphi$. This concludes the proof.
\end{proof}

\begin{os}
\normalfont
For $n=2$ we have that
\begin{align*}
\mathfrak{U}_\beta^{2}(\mathbf{x}, t) = & \int_0^t \frac{1}{s^{\beta}}W_{-\beta, 1-\beta}\left(- \frac{a_1 x_1}{s^{\beta}}\right) \frac{1}{(t-s)^{\beta}}W_{-\beta, 1-\beta}\left(-\frac{a_2 x_2}{(t-s)^{\beta}} \right)\, ds\\
= & \sum_{i=0}^{\infty}\sum_{j=0}^{\infty} \frac{(-a_1 x_1)^i}{i!\, \Gamma(1 -\beta i - \beta)}\frac{(- a_2 x_2)^j}{j!\, \Gamma(1 -\beta j - \beta)}\int_0^t s^{- \beta i - \beta} (t-s)^{-\beta j - \beta}\, ds\\
= & \sum_{i=0}^{\infty}\sum_{j=0}^{\infty} \frac{(- a_1 x_1)^i}{i!\, \Gamma(1 -\beta i - \beta)}\frac{(- a_2 x_2)^j}{j!\, \Gamma(1 -\beta j - \beta)} t^{1-\beta(i+j+2)} B(1-\beta i - \beta, 1-\beta j -\beta)\\
= & \sum_{i=0}^{\infty}\sum_{j=0}^{\infty} \frac{(- a_1 x_1)^i}{i!}\frac{(- a_2 x_2)^j}{j!} \frac{t^{1-\beta(i+j+2)} }{\Gamma(2-2\beta -\beta i - \beta j)}
\end{align*}
where
\begin{equation*}
B(w_1, w_2) = \frac{\Gamma(w_1)\Gamma(w_2)}{\Gamma(w_1+w_2)}
\end{equation*}
is the Beta function. From the fact that (see \citet[p. 115]{LE})
$$ \frac{1}{\Gamma(\alpha)} = \frac{1}{2\pi i} \int_{\mathcal{C}} s^{-\alpha} e^s\, ds $$
we can write
\begin{align*}
\mathfrak{U}_\beta^{2}(\mathbf{x}, t) = &  t^{1- 2 \beta} \sum_{l=0}^{\infty}\sum_{j=0}^{\infty} \frac{(- a_1 x_1 /t^\beta )^l}{l!}\frac{(- a_2 x_2 / t^{\beta})^j}{j!} \frac{1}{2\pi i} \int_{\mathcal{C}} s^{2\beta + \beta l + \beta j -2} e^s\, ds\\
= & t^{1- 2 \beta} \frac{1}{2\pi i} \int_{\mathcal{C}} \sum_{l=0}^{\infty}\sum_{j=0}^{\infty} \frac{(- a_1 x_1 s^\beta /t^\beta)^l}{l!}\frac{(- a_2 x_2 s^{\beta}/t^\beta)^j}{j!} \, s^{2\beta -2} e^s\, ds\\
= & t^{1- 2 \beta} \frac{1}{2\pi i} \int_{\mathcal{C}} e^{- \frac{(a_1 x_1 + a_2 x_2)}{t^\beta} s^{\beta}} s^{2\beta -2}\, ds\\
= & t^{1- 2 \beta} \frac{1}{2\pi i} \int_{\mathcal{C}} \sum_{k=0}^{\infty} \frac{(- (a_1 x_1 + a_2 x_2)/t^\beta)^k}{k!} s^{\beta k} s^{2\beta -2}\, ds\\
= & t^{1- 2 \beta} \sum_{k=0}^{\infty} \frac{(- (a_1 x_1 + a_2 x_2)/t^{\beta})^k}{k!\, \Gamma(- \beta k + 2 - 2 \beta)}\\
= & \frac{1}{t^{2\beta - 1}} W_{-\beta, 2- 2\beta}\left( - \frac{a_1 x_1+a_2x_2}{t^\beta} \right)
\end{align*}
which is in accord with \eqref{lawmultiL0}.
\end{os}

\begin{os}
\normalfont
We can obtain the previous result by considering that
\begin{equation*}
\widetilde{\mathfrak{U}_\beta^{2}}(\boldsymbol{\xi}, t) = \int_0^t \widetilde{u_1}(\xi_1, s)\, \widetilde{u_2}(\xi_2, t-s)\, ds
\end{equation*}
where
\begin{equation*}
\widetilde{u\,_{j}}(\xi_j, t) = E_{\beta}(-\xi_j\, t^{\beta}/a_j) / a_j, \quad j=1,2.
\end{equation*}
Thus, we arrive at
\begin{align*}
a_1a_2\int_0^t E_{\beta}(-\xi_2 s^{\beta}/a_2)\, E_{\beta}(-\xi_1 (t-s)^{\beta}/a_1)\, ds = & \int_0^\infty \int_0^\infty  e^{-\gamma_1 - \gamma_2} \Sigma(\gamma_1, \gamma_2)\, \, d\gamma_1\, d\gamma_2
\end{align*}
where
\begin{align*}
\Sigma(\gamma_1, \gamma_2) = & \sum_{l=0}^{\infty}\sum_{j=0}^{\infty} \frac{(-\xi_2 \gamma_2 / a_2 )^l}{l!\Gamma(i\beta + 1)}\frac{(-\xi_1 \gamma_1 /a_1)^j}{j! \Gamma(j\beta +1)} \int_0^t s^{l \beta} (t-s)^{j\beta}\, ds\\
= & \sum_{l=0}^{\infty}\sum_{j=0}^{\infty} \frac{(-\xi_2 \gamma_2 /a_2)^l}{l!\Gamma(i\beta + 1)}\frac{(-\xi_1 \gamma_1 /a_1)^j}{j! \Gamma(j\beta +1)} t^{\beta(l+j) + 1} \int_0^1 s^{l \beta } (1-s)^{j\beta}\, ds \\
= &  \sum_{l=0}^{\infty}\sum_{j=0}^{\infty} \frac{(-\xi_2 \gamma_2 / a_2)^l}{l!\Gamma(l\beta + 1)}\frac{(-\xi_1 \gamma_1 / a_1)^j}{j! \Gamma(j\beta +1)} t^{\beta(l+j)+1} B(l\beta + 1, j\beta +1) \\
= & \sum_{l=0}^{\infty}\sum_{j=0}^{\infty} \frac{(-\xi_2 \gamma_2 /a_2)^l}{l!}\frac{(-\xi_1 \gamma_1 /a_1)^j}{j!}  \frac{t^{\beta(l+j)+1}}{\Gamma(2+ \beta(l+j))}  \\
= &  \sum_{l=0}^{\infty}\sum_{j=0}^{\infty} \frac{(-\xi_2 \gamma_2 /a_2)^l}{l!}\frac{(-\xi_1 \gamma_1 /a_1)^j}{j!} \frac{1}{2\pi i} \int_{\mathcal{C}} s^{-1} e^s \left( t/s \right)^{\beta(l+j)+1} ds\\
= &  \frac{1}{2\pi i} \int_{\mathcal{C}} s^{-1} e^s (t/s) \sum_{l=0}^{\infty}\sum_{j=0}^{\infty} \frac{(-\xi_2 /a_2 \gamma_2 t^{\beta}/s^{\beta})^l}{l!}\frac{(-\xi_1 /a_1 \gamma_1 t^{\beta}/s^{\beta})^j}{j!}\, ds\\
= & \frac{t}{2\pi i} \int_{\mathcal{C}} s^{-2} \exp\left( s-\frac{t^\beta}{s^{\beta}}(\xi_1 \gamma_1 /a_1+ \xi_2 \gamma_2 /a_2)\right) ds.
\end{align*}
By observing that
\begin{align*}
&\frac{t}{2\pi i} \int_{\mathcal{C}} s^{-2} \exp\left( s-\frac{t^\beta}{s^{\beta}}(\xi_1 \gamma_1 /a_1+ \xi_2 \gamma_2 /a_2)\right) ds \\
= & \frac{1}{2\pi i} \int_{\mathcal{C}^{\prime}} e^{st} s^{-2} \exp\left(-\frac{(\xi_1 \gamma_1 / a_1+ \xi_2 \gamma_2 / a_2)}{s^{\beta}}\right) ds
\end{align*}
where $\mathcal{C}^{\prime}$ is similar to $\mathcal{C}$ (see \cite{LE}), we obtain that
\begin{align*}
\widetilde{\mathfrak{U}_\beta^{2}}(\boldsymbol{\xi}, t) = & \frac{1}{a_1a_2}\int_0^\infty \int_0^\infty e^{- \gamma_1 - \gamma_2} \frac{1}{2\pi i} \int_{\mathcal{C}^{\prime}} e^{st} s^{-2} \exp\left(-\frac{(\xi_1 \gamma_1 /a_1+ \xi_2 \gamma_2 / a_2)}{s^{\beta}}\right) ds\, d\gamma_1\, d\gamma_2\\
= & \frac{1}{2\pi i} \int_{\mathcal{C}^{\prime}} e^{st} s^{-2} \frac{1/a_1}{1 + \frac{\xi_1 /a_1}{s^\beta}} \frac{1/a_2}{1 + \frac{\xi_2 /a_2}{s^\beta}}ds\\
= & \frac{1}{2\pi i} \int_{\mathcal{C}^{\prime}} e^{st} s^{2\beta -2} \frac{1/a_1}{s^\beta + \xi_1 /a_1} \frac{1/a_2}{s^\beta + \xi_2 /a_2}.
\end{align*}
By inverting the Laplace transform we get the Laplace convolution of two Mittag-Leffler functions $\widetilde{\mathfrak{U}^2_\beta}$. 
\end{os}

\section{Second order operators}
\label{secIV}

In the previous sections we have studied the fractional power of the first order operator \eqref{dirder}. Here we present some result on the squared power of \eqref{dirder} and  the fractional power $\vartheta \in (0,1)$ of the negative Laplace operator $-\triangle$ which has been first investigated by \citet{Boch49, Feller52} and many other authors so far. 

Let us introduce the fractional power of \eqref{dirder} formally written as
\begin{equation} 
|\mathbf{a} \cdot \nabla |^{\alpha} = \left( (\mathbf{a} \cdot \nabla)^2 \right)^{\frac{\alpha}{2}} =\left( \sum_{i=0}^{n} \sum_{j=0}^n a_{ij}\frac{\partial^2}{\partial x_i \partial x_j} \right)^{\frac{\alpha}{2}}
\end{equation} 
with $a_{ij}=a_i a_j$, $1\leq i,j \leq n$ and acting on the space of twice differentiable functions defined on $\mathbb{R}^n$. We present some results for $\alpha =2$. As pointed out by \citet{MBB99, MerMorSch04}, for $\alpha=2$, formula \eqref{furMeer} leads to the directional derivative
\begin{equation}
\mathbb{D}^2_M f(\mathbf{x}) = \nabla A \nabla f(\mathbf{x}) = \sum_{i=1}^n \sum_{j=1}^n a_{i,j} \frac{\partial^2 f(\mathbf{x})}{\partial x_i \partial x_j}
\end{equation}
where $A=\{a_{ij}\}_{1 \leq i,j \leq n}$ and 
$$a_{ij} = \int_{\|\boldsymbol{\theta}\|=1} \theta_i \theta_j M(d\boldsymbol{\theta}). $$

\begin{te}
The solution to
\begin{equation} 
\left( \frac{\partial^\beta}{\partial t^\beta} - (\mathbf{a} \cdot \nabla)^2 \right) g(\mathbf{x}, t)=0, \quad (\mathbf{x},t) \in \mathbb{R}^n \times (0, +\infty) \label{op2Ord}
\end{equation}
subject to the initial condition $g(\mathbf{x}, 0)=\delta(\mathbf{x})$ is given by
\begin{equation}
g(\mathbf{x}, t) = \frac{1}{t^{\beta/2}}W_{-\frac{\beta}{2}, 1- \frac{\beta}{2}} \left(- \frac{|\mathbf{a} \cdot \mathbf{x}|}{t^{\beta/2}} \right). \label{sol2Ord}
\end{equation}
Furthermore, formula \eqref{sol2Ord} can be written as
\begin{equation}
g(\mathbf{x}, t) = \int_0^\infty \frac{e^{-\frac{(\mathbf{a} \cdot \mathbf{x})^2}{4s}}}{\sqrt{4\pi s}}l_{\beta}(s, t)ds \label{sol2Ord2}
\end{equation}
where
\begin{equation*}
l_\beta(s, t) = \frac{1}{t^\beta} W_{-\beta, 1-\beta}\left(- \frac{s}{t^\beta} \right).
\end{equation*}
\end{te}
\begin{proof}
We recall that
\begin{equation*}
\int_0^\infty e^{-\lambda x} l_{\beta}(x, t)dx = \lambda^{\beta -1} e^{-x \lambda^\beta}.
\end{equation*}
From the fact that (see \citet{LE})
\begin{equation*} 
 \frac{1}{2}\left( \frac{z^2}{4 \gamma}\right)^\frac{\nu}{2} \int_0^\infty e^{-s\gamma -\frac{z^2}{4s}} s^{-\nu-1}ds =K_{\nu}\left(|z| \gamma^{\frac{1}{2}}\right) , \quad z \in \mathbb{R}, \; \gamma >0
\end{equation*}
where $K_\nu$ is the modified Bessel function of imaginary argument and (\cite{LE})
\begin{equation*}
K_{\frac{1}{2}}(z) = \sqrt{\frac{\pi}{2z}}e^{-z}
\end{equation*} 
we obtain that
\begin{align*}
& \int_0^\infty e^{-\lambda t} \int_0^\infty \frac{e^{-\frac{(\mathbf{a} \cdot \mathbf{x})^2}{4s}}}{\sqrt{4\pi s}}l_{\beta}(s, t)ds \, dt \\
= & \frac{\lambda^{\beta -1}}{\sqrt{4\pi}} \int_0^\infty s^{\frac{1}{2}-1}e^{-\frac{(\mathbf{a} \cdot \mathbf{x})^2}{4s} - s \lambda^{\beta}}ds\\
= & \frac{\lambda^{\beta -1}}{\sqrt{\pi}} \left( \frac{(\mathbf{a} \cdot \mathbf{x})^2}{4 \lambda^{\beta}} \right)^{\frac{1}{4}} K_{\frac{1}{2}}\left( |\mathbf{a} \cdot \mathbf{x}| \lambda^{\frac{\beta}{2}}\right)\\
= & \frac{\lambda^{\beta -1}}{\sqrt{\pi}} \left( \frac{(\mathbf{a} \cdot \mathbf{x})^2}{4 \lambda^{\beta}} \right)^{\frac{1}{4}} \sqrt{\frac{\pi}{2}} \left( |\mathbf{a} \cdot \mathbf{x}| \lambda^{\frac{\beta}{2}}\right)^{-\frac{1}{2}} \exp\left( - |\mathbf{a} \cdot \mathbf{x}| \lambda^{\frac{\beta}{2}}\right)\\
= & \frac{1}{2} \lambda^{\frac{\beta}{2} - 1} \exp\left( - |\mathbf{a} \cdot \mathbf{x}| \lambda^{\frac{\beta}{2}}\right)
\end{align*}
which coincides with
\begin{align*}
& \int_0^\infty e^{-\lambda t} \frac{1}{2t^{\beta/2}}W_{-\frac{\beta}{2}, 1- \frac{\beta}{2}} \left(- \frac{|\mathbf{a} \cdot \mathbf{x}|}{t^{\beta/2}} \right) dt \\
= & \frac{1}{2} \sum_{k=0}^{\infty}  \frac{( - |\mathbf{a} \cdot \mathbf{x}|)^k}{k!}  \int_0^\infty  \frac{e^{-\lambda t} t^{-\beta/2 -\frac{\beta k}{2}}\, dt}{\Gamma\left(-\frac{\beta k}{2} + 1 - \frac{\beta}{2}\right)} \\
= & \frac{\lambda^{\frac{\beta}{2} -1}}{2}\sum_{k=0}^{\infty}  \frac{( - |\mathbf{a} \cdot \mathbf{x}| \lambda^{\frac{\beta}{2}})^k}{k!}\\
= & \frac{1}{2} \lambda^{\frac{\beta}{2} - 1} \exp\left( - |\mathbf{a} \cdot \mathbf{x}| \lambda^{\frac{\beta}{2}}\right).
\end{align*}
This prove that formula \eqref{sol2Ord} can be rewritten as \eqref{sol2Ord2}. Formula \eqref{sol2Ord2} can be rewritten as follows
\begin{align}
g(\mathbf{x}, t) = & \int_0^\infty \left( \frac{1}{2\pi} \int_{-\infty}^{+\infty} e^{-i \gamma \, (\mathbf{a}\cdot \mathbf{x})} e^{-\gamma^2 s}d\gamma \right) \, l_{\beta}(s, t)ds \notag \\
= & \frac{1}{2\pi} \int_{-\infty}^{+\infty} e^{-i \gamma \, (\mathbf{a}\cdot \mathbf{x})}\left( \int_0^\infty  e^{-\gamma^2 s}l_{\beta}(s, t)ds \right)d\gamma \notag \\
= & \frac{1}{2\pi} \int_{-\infty}^{+\infty} e^{-i \gamma \, (\mathbf{a}\cdot \mathbf{x})} E_{\beta}\left( -t^\beta \gamma^2 \right) d\gamma. \label{formulagP}
\end{align}

\begin{equation*}
\frac{\partial^\beta}{\partial t^\beta} \frac{1}{2\pi} \int_{-\infty}^{+\infty} e^{-i \gamma \, (\mathbf{a}\cdot \mathbf{x})} E_{\beta}\left( -t^\beta \gamma^2 \right) d\gamma = \frac{1}{2\pi} \int_{-\infty}^{+\infty} e^{-i \gamma \, (\mathbf{a}\cdot \mathbf{x})} (-\gamma^2) E_{\beta}\left( -t^\beta \gamma^2 \right) d\gamma
\end{equation*}
which coincides with
\begin{equation*}
(\mathbf{a} \cdot \nabla)^2 \frac{1}{2\pi} \int_{-\infty}^{+\infty} e^{-i \gamma \, (\mathbf{a}\cdot \mathbf{x})} E_{\beta}\left( -t^\beta \gamma^2 \right) d\gamma
\end{equation*}
by considering that $\|\mathbf{a}\|=1$ and provided that $\gamma^2 E_{\beta}\left( -t^\beta \gamma^2 \right) \in L^1(d\gamma)$. By taking the Fourier transform of \eqref{formulagP} we obtain
\begin{align*}
\widehat{g}(\boldsymbol{\xi}, t) = & \int_{\mathbb{R}^n} e^{i \boldsymbol{\xi}\cdot \mathbf{x}} g(\mathbf{x}, t)\, d\mathbf{x} = \int_{-\infty}^{+\infty} \delta(\boldsymbol{\xi} - \gamma \mathbf{a}) E_{\beta}\left( -t^\beta \gamma^2 \right) d\gamma
\end{align*}
where $\boldsymbol{\xi} = \gamma \mathbf{a}$ if and only if $\mathbf{a}\cdot \boldsymbol{\xi} = \gamma \|\mathbf{a}\|^2= \gamma$ and therefore we get
\begin{align*}
\widehat{g}(\boldsymbol{\xi}, t) = &  \int_{-\infty}^{+\infty} \delta(\mathbf{a}\cdot \boldsymbol{\xi} - \gamma) E_{\beta}\left( -t^\beta \gamma^2 \right) d\gamma = E_{\beta}\left( -t^\beta (\mathbf{a}\cdot \boldsymbol{\xi})^2 \right).
\end{align*}
From the fact that
\begin{equation*}
\int_{-\infty}^{+\infty} e^{i \,\boldsymbol{\xi} \cdot \mathbf{x}}\, (\mathbf{a}\cdot \nabla)^2 u(x)\, dx = \left( - i\, \mathbf{a} \cdot \boldsymbol{\xi} \right)^2 \int_{-\infty}^{+\infty} e^{i \,\boldsymbol{\xi} \cdot \mathbf{x}}u(x)\, dx = -\left( \mathbf{a} \cdot \boldsymbol{\xi} \right)^2\, \widehat{u}(\boldsymbol{\xi})
\end{equation*}
$u \in Dom\, (\mathbf{a}\cdot \nabla)$, we get that the equation \eqref{op2Ord} becomes
\begin{equation*}
\frac{\partial^\beta \widehat{g}}{\partial t^\beta}(\boldsymbol{\xi}, t) = - \left( \mathbf{a} \cdot \boldsymbol{\xi} \right)^2\, \widehat{g}(\boldsymbol{\xi}, t)
\end{equation*}
which leads to
\begin{equation*}
\widehat{g}(\boldsymbol{\xi}, t) = E_{\beta}\left(-t^\beta \, \left( \mathbf{a} \cdot \boldsymbol{\xi} \right)^2 \right)
\end{equation*}
by taking into account the initial condition $g(\mathbf{x}, 0) = \delta(\mathbf{x})$. We recall that the Mittag-Leffler function is an eigenfunction for the D-C fractional  derivative $\frac{\partial^\beta}{\partial t^\beta}$ (see formula \eqref{relaxEq}). This concludes the proof.
\end{proof}

\begin{os}
\normalfont
We notice that
\begin{equation*}
2 \frac{e^{-\frac{x^2}{4s}}}{\sqrt{4\pi s}} = \frac{1}{t^{1/2}}W_{-\frac{1}{2}, 1-\frac{1}{2}}\left( -\frac{|x|}{t^{1/2}} \right) = l_{\frac{1}{2}}(|x|, t)
\end{equation*}
and therefore we can write the function \eqref{sol2Ord} as follows
\begin{equation*}
g(\mathbf{x}, t) = \int_{0}^{\infty} l_{\frac{1}{2}}(|\mathbf{a}\cdot \mathbf{x}|, s) \, l_{\beta}(s, t)\, ds.
\end{equation*}
Furthermore, the one-dimensional version of \eqref{op2Ord} is the fractional diffusion equation
\begin{equation*} 
\left( \frac{\partial^\beta}{\partial t^\beta} - \frac{\partial^2}{\partial x^2} \right) g(x, t)=0, \quad x \in \mathbb{R},\; t>0 
\end{equation*}
subject to the initial condition $g(x, 0)=\delta(x)$ whose stochastic solution is represented by the subordinated Brownian motion
\begin{equation*}
B(\mathfrak{L}^\beta_t), \quad t>0
\end{equation*}
as shown by \citet{OB09, BMN09ann}. Furthermore, we notice that $g(\mathbf{x}, t) = l_\frac{\beta}{2}(|\mathbf{a}\cdot \mathbf{x}|, t)$ and solves \eqref{op2Ord} whereas, $l_\frac{\beta}{2}(\mathbf{a}\cdot \mathbf{x}, t)$ solves
\begin{equation} 
\left( \frac{\partial^\frac{\beta}{2}}{\partial t^\frac{\beta}{2}} - (\mathbf{a} \cdot \nabla) \right) l_\frac{\beta}{2}(\mathbf{a}\cdot \mathbf{x}, t)=0, \quad (\mathbf{x},t) \in \mathbb{R}^n_+ \times (0, +\infty)
\end{equation}
as shown in Theorem \ref{mainTheo}.
\end{os}

Let us consider $u \in \mathscr{S}$ where $\mathscr{S}$ is the Schwartz space of rapidly decaying $C^\infty$ functions in $\mathbb{R}^n$. The fractional power $\vartheta \in (0,1)$ of the negative Laplace operator $-\triangle$ is defined as follows
\begin{align*}
(-\triangle)^\vartheta u(\mathbf{x}) = & C\, p.v.\, \int_{\mathbb{R}^n} \frac{u(\mathbf{x}) - u(\mathbf{y})}{|\mathbf{x} - \mathbf{y}|^{n+2\vartheta}} d\mathbf{y}\\
= & - \frac{C}{2} \int_{\mathbb{R}^n} \frac{u(\mathbf{x} + \mathbf{y}) + u(\mathbf{x}- \mathbf{y}) - 2u(\mathbf{x})}{|\mathbf{y}|^{n+2\vartheta}} d\mathbf{y}
\end{align*}
where $C$ is a constant depending on $(\vartheta,n)$ and, "$p.v.$" stands for "principal value". An alternative way (in the space of Fourier transforms) to define the fractional power of $-\triangle$ is 
\begin{equation*} 
 -(-\triangle)^\vartheta u(\mathbf{x}) = \frac{1}{2\pi} \int_{\mathbf{R}} e^{-i \boldsymbol{\xi} \cdot \mathbf{x}} \|\boldsymbol{\xi} \|^{2\vartheta} \widehat{u}(\boldsymbol{\xi})\, d\boldsymbol{\xi}.
\end{equation*} 
Indeed, we have that
\begin{equation}
\int_{\mathbb{R}^n} e^{i\boldsymbol{\xi}\cdot \mathbf{x}} (-\triangle)^\vartheta u(\mathbf{x}) \, d\mathbf{x} = \| \boldsymbol{\xi}\|^{2\vartheta} \, \widehat{u}(\boldsymbol{\xi}).\label{furFLAP}
\end{equation}

The solution to the fractional equation
\begin{equation}
\frac{\partial u}{\partial t}(\mathbf{x}, t) = - (-\triangle)^\vartheta u(\mathbf{x}, t), \quad \mathbf{x} \in \mathbb{R}^n,\; t>0 \label{fracLapPDE}
\end{equation}
in the Sobolev space
\begin{equation}
Dom\left(-(- \triangle)^\vartheta \right) = \left\lbrace u \in L^2(\mathbb{R}^n)\, : \int_{\mathbb{R}^n} (1+\|\boldsymbol{\xi}\|^{2\vartheta}) |\widehat{u}(\boldsymbol{\xi})|^2 d\boldsymbol{\xi} < \infty \right\rbrace
\end{equation}
and subject to the initial condition $u_0(\mathbf{x}) = \delta(\mathbf{x})$, represents the law of an isotropic $\mathbb{R}^n$-valued stable process, say $\mathbf{S}_{2\vartheta}(t)$, $t>0$, with characteristic function
\begin{equation}
\mathbb{E} \,e^{i\boldsymbol{\xi} \cdot \mathbf{S}_{2\vartheta}(t)} = e^{- t \| \boldsymbol{\xi} \|^{2\vartheta}}.
\end{equation}
Thus, from our viewpoint, the isotropic stable process $\mathbf{S}_{2\vartheta}(t)$, $t>0$, $\vartheta \in (0,1)$ is the stochastic solution to \eqref{fracLapPDE} subject to the initial datum $u_0=\delta$. Obviously, for $\vartheta=1$, we obtain $\mathbf{S}_{2}(t)=\mathbf{B}(t)$, $t>0$ which is the $n$-dimensional Brownian motion. 

We now present the main result of this section.
\begin{te}
For $\vartheta \in (0,1)$, $\alpha, \beta \in (0,1)$, $\mathbf{a}\in \mathbb{R}^n_+$ such that $\| \mathbf{a}\| =1$, the solution $w=w(\mathbf{x},t)$ to the fractional equation
\begin{equation}
\frac{\partial^\beta w}{\partial t^\beta} = -(- \triangle)^\vartheta w - (\mathbf{a}\cdot \nabla)^\alpha w\label{eqDifAdv}
\end{equation}
for $(\mathbf{x}, t) \in \mathbb{R}^n \times (0, +\infty)$, subject to the initial condition $w_0=\delta$, is given by
\begin{equation}
w(\mathbf{x}, t) = \int_0^\infty dz \int_{\mathbb{R}^n_+} d\mathbf{s} \,\mathcal{T}_{2\vartheta}(\mathbf{x} - \mathbf{s}, z) \, h_\alpha(\mathbf{a}\cdot \mathbf{s}, z)\, l_\beta(z, t) \label{solDifAdv}
\end{equation}
where $\mathcal{T}_{2\vartheta}(\mathbf{x}, t)$, $\mathbf{x} \in \mathbb{R}^n$, $t>0$ is the law of the isotropic stable L\'evy process $\mathbf{S}_{2\vartheta}(t)$, $t>0$; $h_\alpha(x, t)$, $x \in [0, +\infty)$, $t>0$, is the law of the stable subordinator $\mathfrak{H}^\alpha_t$, $t>0$; $l_\beta(x, t)$, $x \in (0, +\infty)$, $t>0$, is the law of the inverse process $\mathfrak{L}^\beta_t$, $t>0$.   
\end{te}
\begin{proof}
First, from formulae \eqref{deffracdirder} and \eqref{opRule1}, we evaluate the Fourier transform of a function $u \in L^1(\mathbb{R}^n)$ given by
\begin{align}
\int_{\mathbb{R}^n} e^{i \boldsymbol{\xi} \cdot \mathbf{x}} (\mathbf{a}\cdot \nabla)^\alpha u(\mathbf{x})d\mathbf{x} = & \int_{\mathbb{R}^n} e^{i \boldsymbol{\xi} \cdot \mathbf{x}} \frac{(\mathbf{a}\cdot \nabla)}{\Gamma(1-\alpha)} \int_{0}^\infty ds\, s^{-\alpha} e^{-s (\mathbf{a}\cdot \nabla)} u(\mathbf{x})d\mathbf{x}\notag \\
= & \frac{(-i\mathbf{a}\cdot \xi)}{\Gamma(1-\alpha)} \int_{0}^\infty ds\, s^{-\alpha} \int_{\mathbb{R}^n} e^{i \boldsymbol{\xi} \cdot \mathbf{x}} e^{-s (\mathbf{a}\cdot \nabla)} u(\mathbf{x})d\mathbf{x}\notag \\
= & \frac{(-i\mathbf{a}\cdot \xi)}{\Gamma(1-\alpha)} \int_{0}^\infty ds\, s^{-\alpha} \int_{\mathbb{R}^n} e^{i \boldsymbol{\xi} \cdot \mathbf{x}} u(\mathbf{x}-s\mathbf{a}) d\mathbf{x}\notag \\
= & \frac{(-i\mathbf{a}\cdot \xi)}{\Gamma(1-\alpha)} \int_{0}^\infty ds\, s^{-\alpha} e^{i \, s \, \mathbf{a}\cdot \boldsymbol{\xi}} \int_{\mathbb{R}^n} e^{i \boldsymbol{\xi} \cdot \mathbf{x}} u(\mathbf{x}) d\mathbf{x}\notag \\
= & (-i\mathbf{a}\cdot \xi)^\alpha \widehat{u}(\boldsymbol{\xi})\label{furAnalphaR}
\end{align}
(where we used \eqref{LapdistrZ}) which is in accord with \eqref{furAnalpha} for functions in $L^1(\mathbb{R}^n_+)$. From \eqref{furFLAP} and \eqref{furAnalphaR}, by passing to the Fourier transform, the equation \eqref{eqDifAdv} takes the form
\begin{equation}
\frac{\partial^\beta \widehat{w}}{\partial t^\beta}(\boldsymbol{\xi}, t) = - \| \boldsymbol{\xi} \|^{2\vartheta} \widehat{w}(\boldsymbol{\xi}, t) - (-i \mathbf{a}\cdot \boldsymbol{\xi})^\alpha \widehat{w}(\boldsymbol{\xi}, t). \label{coincFurDifAdv}
\end{equation}
From the fact that
\begin{equation*}
\int_{\mathbb{R}^n} e^{i\boldsymbol{\xi} \cdot \mathbf{x}}  \mathcal{T}_{2\vartheta}(\mathbf{x} - \mathbf{s}, z)\, d\mathbf{x} = e^{i \boldsymbol{\xi}\cdot \mathbf{s}} \int_{\mathbb{R}^n} e^{i\boldsymbol{\xi} \cdot \mathbf{x}}  \mathcal{T}_{2\vartheta}(\mathbf{x} , z)\, d\mathbf{x}= e^{i \boldsymbol{\xi}\cdot \mathbf{s}} \,  e^{-z \| \boldsymbol{\xi} \|^{2\vartheta}}
\end{equation*}
we get that
\begin{align*}
\widehat{w}(\boldsymbol{\xi}, t) = & \int_{\mathbb{R}^n} e^{i\boldsymbol{\xi} \cdot \mathbf{x}} w(\mathbf{x}, t)\, d\mathbf{x}\\
= & \int_0^\infty dz  \, e^{-z \| \boldsymbol{\xi} \|^{2\vartheta}} \int_{\mathbb{R}^n_+} d\mathbf{s} \, e^{i \boldsymbol{\xi}\cdot \mathbf{s}} \, h_\alpha(\mathbf{a}\cdot \mathbf{s}, z)\, l_\beta(z, t)
\end{align*}
From Remark \ref{RemfurRemH} (formula \eqref{furRemH} in particular) we have that
\begin{align*}
\widehat{w}(\boldsymbol{\xi}, t) = & \int_0^\infty dz  \, e^{-z \| \boldsymbol{\xi} \|^{2\vartheta}} \, e^{-z(-i \mathbf{a}\cdot \boldsymbol{\xi})^\alpha}\, l_\beta(z, t)
\end{align*}
and, from \eqref{lapLspace}, 
\begin{equation}
\widehat{w}(\boldsymbol{\xi}, t) = E_{\beta}\left(-t^\beta  \|\boldsymbol{\xi} \|^{2\vartheta} -t^\beta (-i \mathbf{a}\cdot \boldsymbol{\xi})^\alpha  \right). \label{fursolDifAdv}
\end{equation}
From the fact that 
\begin{equation*}
\frac{\partial^\beta E_\beta}{\partial t^\beta}(-t^\beta \zeta) = -\zeta\, E_\beta(-t^\beta \zeta), \quad \zeta >0
\end{equation*}
which means that the Mittag-Leffler is an eigenfunction for the Dzerbayshan-Caputo time-fractional derivative, we arrive at 
\begin{equation*}
\frac{\partial^\beta \widehat{w}}{\partial t^\beta}(\boldsymbol{\xi}, t) = -\left( \|\boldsymbol{\xi} \|^{2\vartheta} + (-i \mathbf{a}\cdot \boldsymbol{\xi})^\alpha  \right) \, \widehat{w} (\boldsymbol{\xi}, t) 
\end{equation*}
which coincides with \eqref{coincFurDifAdv}. This concludes the proof.
\end{proof}

\begin{os}
\normalfont
The distribution \eqref{solDifAdv} can be regarded as the law of a subordinated $\mathbb{R}^n$-valued stable process with stable subordinated  drift given by
\begin{equation*}
\mathbf{W}(t) = \mathbf{S}_{2\vartheta}(\mathfrak{L}^\beta_t) + \mathbf{a}\, \mathfrak{H}^\alpha_{\mathfrak{L}^\beta_t}, \quad t>0.
\end{equation*} 
The characteristic function is given by
\begin{align*}
\mathbb{E}\, e^{i \boldsymbol{\xi}\cdot \mathbf{W}(t)} = & \mathbb{E}\left[ \mathbb{E}\, e^{i \boldsymbol{\xi}\cdot \mathbf{S}_{2\vartheta}(T) + i \boldsymbol{\xi}\cdot \mathbf{a}\,\mathfrak{H}^\alpha_T}\Bigg| T=\mathfrak{L}^\beta_t \right]\\
=&  \mathbb{E}\left[ e^{-T \| \boldsymbol{\xi} \|^{2\vartheta}} \mathbb{E} \, e^{-(-i \boldsymbol{\xi}\cdot \mathbf{a}) \, \mathfrak{H}^\alpha_T} \Bigg| T=\mathfrak{L}^\beta_t \right]\\
= & (\textrm{by } \eqref{lapHspace} )\\
= & \mathbb{E}\left[ e^{-T \| \boldsymbol{\xi} \|^{2\vartheta} - T (-i \boldsymbol{\xi}\cdot \mathbf{a})^\alpha } \Bigg| T=\mathfrak{L}^\beta_t \right]\\ 
= & \mathbb{E} \, e^{-(\| \boldsymbol{\xi} \|^{2\vartheta} + (-i \boldsymbol{\xi}\cdot \mathbf{a})^\alpha )\mathfrak{L}^\beta_t}\\
= & (\textrm{by } \eqref{lapLspace} )\\
= &E_\beta \left(-t^\beta \, \| \boldsymbol{\xi} \|^{2\vartheta} - t^\beta \, (-i \boldsymbol{\xi}\cdot \mathbf{a})^\alpha \right)
\end{align*}
which coincides with \eqref{fursolDifAdv}.
\end{os}

\begin{os}
\normalfont
We remark that, for $\alpha=\beta$, the stochastic solution to \eqref{eqDifAdv} becomes
\begin{equation*}
\mathbf{S}_{2\vartheta}(\mathfrak{L}^\beta_t) + \mathbf{a}\, t \, \frac{{_1\mathfrak{H}^\beta}}{{_2\mathfrak{H}^\beta}} , \quad t>0.
\end{equation*} 
where the ratio of two independent stable subordinators ${_j\mathfrak{H}^\beta}$, $j=1,2$, is independent from $t$ and possesses distribution 
\begin{equation}
Pr\{ {_1\mathfrak{H}^\beta}/ {_2\mathfrak{H}^\beta} \in dx \}/dx = \frac{\sin \beta \pi}{\pi} \frac{x^{\beta -1}}{x^{2\beta} + 2x^\beta \cos \beta \pi + 1}, \quad x\geq 0, \; t>0. \label{ratioSSdist}
\end{equation}
This has been pointed out also in Remark \ref{Remaratio}.
\end{os}

\begin{os}
\normalfont
For $\vartheta=1$ in \eqref{eqDifAdv}, we immediately have that the fractional equation
\begin{equation*}
\frac{\partial^\beta w}{\partial t^\beta} = \triangle w - (\mathbf{a}\cdot \nabla)^\alpha w
\end{equation*}
has a solution which is the law of the process
\begin{equation*}
\mathbf{B}(\mathfrak{L}^\beta_t) + \mathbf{a}\, \mathfrak{H}^\alpha_{\mathfrak{L}^\beta_t}, \quad t>0
\end{equation*}
($\mathbf{B}$ is a  Brownian motion) and can be written as follows
\begin{equation*}
w(\mathbf{x}, t) = \int_0^\infty dz \int_{\mathbb{R}^n_+} d\mathbf{s} \,\frac{e^{-\frac{\|\mathbf{x} - \mathbf{s}\|^2}{4z}}}{\sqrt{4\pi z}} h_\alpha(\mathbf{a}\cdot \mathbf{s}, z)\, l_\beta(z, t).
\end{equation*}
For $\alpha \to 1$ we have that $\mathfrak{H}^\alpha_t \to t$ which is the elementary subordinator (obviously $h_\alpha(x, t) \to \delta(x-s)$) and therefore the solution to
\begin{equation*}
\frac{\partial^\beta w}{\partial t^\beta} = \triangle w - (\mathbf{a}\cdot \nabla)w
\end{equation*}
is written as
\begin{align*}
w(\mathbf{x}, t) = & \int_0^\infty dz \int_{\mathbb{R}^n_+} d\mathbf{s} \,\frac{e^{-\frac{\|\mathbf{x} - \mathbf{s}\|^2}{4z}}}{\sqrt{4\pi z}} \delta(\mathbf{a}\cdot \mathbf{s}- z)\, l_\beta(z, t)\\
= &  \int_0^\infty dz \int_{\mathbb{R}^n_+} d\mathbf{s} \,\frac{e^{-\frac{\|\mathbf{x} - \mathbf{s}\|^2}{4z}}}{\sqrt{4\pi z}} \delta(\mathbf{s}- \mathbf{a}z)\, l_\beta(z, t)
\end{align*}
where we used the fact that $\mathbf{a}\cdot \mathbf{s}=z \in \mathbb{R}_{+}$ iff $\mathbf{s}= \mathbf{a}z$ or, equivalently
\begin{align*}
w(\mathbf{x}, t) = & \int_0^\infty dz \int_{\mathbb{R}^n_+} d\mathbf{s} \,\frac{e^{-\frac{\|\mathbf{x} - \mathbf{s}\|^2}{4z}}}{\sqrt{4\pi z}} \delta(\mathbf{a}\cdot \mathbf{s}- z)\, l_\beta(z, t)\\
= & \int_{\mathbb{R}^n_+} d\mathbf{s} \,\frac{e^{-\frac{\|\mathbf{x} - \mathbf{s}\|^2}{4 (\mathbf{a}\cdot \mathbf{s})}}}{\sqrt{4\pi (\mathbf{a}\cdot \mathbf{s})}} l_\beta(\mathbf{a}\cdot \mathbf{s}, t).
\end{align*}
Finally, for $\vartheta=1$ and $\alpha=1$, we get that 
\begin{align*}
w(\mathbf{x}, t) = &  \int_0^\infty dz  \,\frac{e^{-\frac{\|\mathbf{x} - \mathbf{a}z\|^2}{4z}}}{\sqrt{4\pi z}}\, l_\beta(z, t)
\end{align*}
is the law of the $n$-dimensional subordinated Brownian motion with subordinated drift
\begin{equation*}
\mathbf{B}(\mathfrak{L}^\beta_t) + \mathbf{a}\, \mathfrak{L}^\beta_t, \quad t>0.
\end{equation*}
If $\beta \to 1$, then $\mathfrak{L}^\beta_t \to t$ and $l_\beta(x, t)\to \delta(x, t)$. The solution to
\begin{equation*}
\frac{\partial w}{\partial t} = \triangle w - (\mathbf{a}\cdot \nabla)w
\end{equation*}
is given by
\begin{equation*}
w(\mathbf{x}, t) = \frac{e^{-\frac{\|\mathbf{x} - \mathbf{a}t\|^2}{4t}}}{\sqrt{4\pi t}}
\end{equation*}
which is the law of
\begin{equation*}
\mathbf{B}(t) + \mathbf{a}\, t, \quad t>0.
\end{equation*}
\end{os}

\begin{os}
\normalfont
For the sake of completeness we also observe that, for $\alpha=1$, equation \eqref{eqDifAdv} becomes
\begin{equation}
\frac{\partial^\beta u}{\partial t^\beta} = - (-\triangle)^\vartheta u - (\mathbf{a}\cdot \nabla)u
\end{equation} 
whose stochastic solution is given by the subordinated stable process with drift
\begin{equation}
\mathbf{S}_{2\vartheta}(\mathfrak{L}^\beta_t) + \mathbf{a} \mathfrak{L}^\beta_t, \quad t>0.
\end{equation}
Indeed, for $\alpha \to 1$, we have that $\mathfrak{H}^\alpha_t \to t$.
\end{os}

\begin{os}
\normalfont
We remark that (\cite{SzWe02,Traple96}) the transport equation 
\begin{equation*}
\frac{\partial u}{\partial t} = Au - \lambda u + \lambda Ku
\end{equation*}
where 
$$Au=- \sum_{k=1}^n \frac{\partial}{\partial x_k} \left( a(x)\, u \right)$$
and $K$ is the Frobenius-Perron operator corresponding to the transformation $T(x)=x+f(x)$ has a solution which is the law of the solution to the Poisson ($\mathbf{N}_t$) driven stochastic differential equation  
\begin{equation}
d\mathbf{X}_t = a(\mathbf{X}_t)dt + f(\mathbf{X}_t)d\mathbf{N}_t.
\end{equation}
\end{os}

\begin{os}
\normalfont
We recall that (see for example \citet{OB09EJP}) the fractional Poisson process $\mathcal{N}^\beta_t = N(\mathfrak{L}^\beta_t)$ is a renewal process with 
\begin{equation*}
Pr\{ \mathcal{N}^\beta_t = k \} = p_k^\beta(t) = \mathbb{E} p_k(\mathfrak{L}^\beta_t)
\end{equation*}
where $p_k(t)$ is the distribution of the Poisson process $N(t)$, $t>0$ and probability generating function written as
\begin{equation}
\mathbb{E} z^{\mathcal{N}^\beta_t} = E_\beta\left(-\lambda (1-z)t^\beta \right).\label{Poispk}
\end{equation}
From \eqref{Poispk} and \eqref{relaxEq} we can write
\begin{equation}
\frac{\partial^\beta}{\partial t^\beta}p_k^\beta(t) = -\lambda \left(p_k^\beta(t) - p^\beta_{k-1}(t) \right)
\end{equation}
Furthermore, we recall that
\begin{equation}
\mathbb{E} z^{N(t)} = e^{-\lambda t (1-z)}
\end{equation}
and
\begin{equation}
\mathbb{E} e^{-i\xi N(t)} = \exp \left(- \xi (1-e^{-i\xi}) \right).
\end{equation}
\end{os}

\begin{te}
The stochastic solution to 
\begin{equation}
\frac{\partial^\beta  \mathsf{w}}{\partial t^\beta}(\mathbf{x}, t) = - \Big( (\mathbf{a}\cdot \nabla)^\alpha + \lambda \left( I - e^{-\mathbf{1}\cdot \nabla} \right) \Big)  \mathsf{w}(\mathbf{x}, t), \quad (\mathbf{x}, t) \in \mathbb{R}^n_+ \times (0, +\infty) \label{eqpoisProc}
\end{equation}
where $\alpha, \beta \in (0, 1]$ and $e^{-\mathbf{1}\cdot \nabla} \mathsf{w}(\mathbf{x}, t) =  \mathsf{w}(\mathbf{x}-\mathbf{1}, t)$ is the shift operator, is the process
\begin{equation}
\mathbf{Y}_t =  \mathbf{N}(\mathfrak{L}^\beta_t) + \mathbf{a}\, \mathfrak{H}^\alpha_{\mathfrak{L}^\beta_t}, \quad t>0.\label{poisProc}
\end{equation}
\end{te}
\begin{proof}
We have that
\begin{align*}
\mathbb{E}e^{i\boldsymbol{\xi} \mathbf{Y}_t} = &  \mathbb{E} \exp\left( - \lambda (1-e^{i\boldsymbol{\xi}})\mathfrak{L}^\beta_t + i \mathbf{a}\cdot \boldsymbol{\xi}\, \mathfrak{H}^\alpha_{\mathfrak{L}^\beta_t} \right)\\
= & \mathbb{E}\exp\left( - \lambda (1-e^{i\boldsymbol{\xi}})\mathfrak{L}^\beta_t - (- i \mathbf{a}\cdot \boldsymbol{\xi})^\alpha \, \mathfrak{L}^\beta_t \right)\\
= & E_\beta \left( - \lambda (1-e^{i\boldsymbol{\xi}})t^\beta - (- i \mathbf{a}\cdot \boldsymbol{\xi})^\alpha \, t^\beta \right)\\
= & \widehat{ \mathsf{w}}(\boldsymbol{\xi}, t)
\end{align*}
is the characteristic function of \eqref{poisProc}. From \eqref{relaxEq} we obtain that
\begin{equation*}
\frac{\partial^\beta \widehat{ \mathsf{w}}}{\partial t^\beta}(\boldsymbol{\xi}, t) =  \left( - \lambda (1-e^{i\boldsymbol{\xi}}) - (- i \mathbf{a}\cdot \boldsymbol{\xi})^\alpha  \right) \widehat{ \mathsf{w}}(\boldsymbol{\xi}, t)
\end{equation*}
where, from Lemma \ref{lemmaFurDirDer}, $\widehat{(\mathbf{a}\cdot \nabla)^\alpha  \mathsf{w}}(\boldsymbol{\xi}, t) = (- i \mathbf{a}\cdot \boldsymbol{\xi})^\alpha \widehat{ \mathsf{w}}(\boldsymbol{\xi}, t)$ and
\begin{align*}
\int_{\mathbb{R}^n} e^{i\boldsymbol{\xi}\cdot \mathbf{x}} \lambda \left( I - e^{-\mathbf{1}\cdot \nabla} \right)  \mathsf{w}(\mathbf{x}, t)\, d\mathbf{x} = & \lambda \left(1 - \int_{\mathbb{R}^n} e^{i\boldsymbol{\xi}\cdot \mathbf{x}}   \mathsf{w}(\mathbf{x}- \mathbf{1}, t)\, d\mathbf{x} \right) \\
= & \lambda \left(1 - e^{i \boldsymbol{\xi}}\right) \widehat{ \mathsf{w}}(\boldsymbol{\xi}, t).
\end{align*}
This shows that \eqref{poisProc} is the solution to the Poisson driven stochastic differential equation whose density law solves \eqref{eqpoisProc}. 
\end{proof}

\begin{os}
\normalfont
Let us consider the equation
\begin{equation}
\frac{\partial^\beta  \mathsf{w}}{\partial t^\beta}(x, t) = - \Big( \partial_x^\alpha + \lambda \frac{\left( I - e^{-\tau \partial_x} \right)}{\tau} \Big)  \mathsf{w}(x, t), \quad (x, t) \in \mathbb{R}_+ \times (0, +\infty) \label{eqpoisProc2}
\end{equation}
which coincides, for $\tau=1$, with \eqref{eqpoisProc} in the one-dimensional case and can be rewritten, for $\tau \neq 0$, as follows
\begin{equation}
\left( \frac{\partial^\beta }{\partial t^\beta} + \partial^\alpha_x \right) \mathsf{w}(x, t) = -\frac{\lambda}{\tau}\left(  \mathsf{w}(x, t) -  \mathsf{w}(x-\tau, t) \right). \label{eqpoisProc3}
\end{equation}
Formula \eqref{eqpoisProc3} is the governing equation of the one-dimensional process
\begin{equation}
Y_t=\tau N(\tau^{-1} \mathfrak{L}^\beta_t) + \mathfrak{H}^\alpha_{\mathfrak{L}^\beta_t}, \quad t>0.
\end{equation}
For $\tau \to 0$, we obtain that
\begin{equation}
\left( \frac{\partial^\beta }{\partial t^\beta} + \partial^\alpha_x \right) \mathsf{w}(x, t) = -\lambda \partial_x  \mathsf{w}(x, t) \label{eqpoisProc4}
\end{equation}
is the governing equation of 
\begin{equation}
Y_t= \lambda t + \mathfrak{H}^\alpha_{\mathfrak{L}^\beta_t}, \quad t>0
\end{equation}
which becomes, for $\alpha = \beta \in (0, 1)$,
\begin{equation}
Y_t= \left( \lambda + \frac{{_1\mathfrak{H}^\alpha_t}}{{_2 \mathfrak{H}^\alpha_t}} \right) t, \quad t>0
\end{equation}
where the ratio of independent stable subordinators has distribution \eqref{ratioSSdist}.
\end{os}


\begin{thebibliography}{30}
\providecommand{\natexlab}[1]{#1}
\providecommand{\url}[1]{\texttt{#1}}
\expandafter\ifx\csname urlstyle\endcsname\relax
  \providecommand{\doi}[1]{doi: #1}\else
  \providecommand{\doi}{doi: \begingroup \urlstyle{rm}\Url}\fi

\bibitem[Anastassiou(2009)]{Anat09}
G.~A. Anastassiou.
\newblock {Distributional Taylor formula}.
\newblock \emph{Nonlinear Analysis}, 70:\penalty0 3195 -- 3202, 2009.

\bibitem[Babusci et~al.(2011)Babusci, Dattoli, and Quattromini]{BDatQ11}
D.~Babusci, G.~Dattoli, and M.~Quattromini.
\newblock Relativistic equations with fractional and pseudo-differential
  operators.
\newblock \emph{Physical Review A}, 83, 2011.

\bibitem[Baeumer et~al.(2009)Baeumer, Meerschaert, and Nane]{BMN09}
B.~Baeumer, M.M. Meerschaert, and E.~Nane.
\newblock Space-time duality for fractional diffusion.
\newblock \emph{J. Appl. Probab.}, 46:\penalty0 1100 -- 1115, 2009.

\bibitem[Balakrishnan(1960)]{Balak60}
A.V. Balakrishnan.
\newblock Fractional powers of closed operators and semigroups generated by
  them.
\newblock \emph{Pacific J. Math.}, 10:\penalty0 419 -- 437, 1960.

\bibitem[Beghin and Orsingher(2009)]{OB09EJP}
L.~Beghin and E.~Orsingher.
\newblock {Fractional Poisson processes and related random motions}.
\newblock \emph{Elect. J. Probab.}, 14:\penalty0 1790 -- 1826, 2009.

\bibitem[Bertoin(1996)]{Btoi96}
J.~Bertoin.
\newblock \emph{{L\'evy Processes}}.
\newblock Cambridge University Press, 1996.

\bibitem[Bochner(1949)]{Boch49}
S.~Bochner.
\newblock Diffusion equation and stochastic processes.
\newblock \emph{Proc. Nat. Acad. Sciences, U.S.A.}, 35:\penalty0 368 -- 370,
  1949.

\bibitem[Butzer and Westphal(2000)]{ButWest2000}
P.~L. Butzer and U.~Westphal.
\newblock \emph{An Introduction to Fractional Calculus}.
\newblock 2000.

\bibitem[D'Ovidio(2010)]{Dov2}
M.~D'Ovidio.
\newblock Explicit solutions to fractional diffusion equations via generalized
  gamma convolution.
\newblock \emph{Elect. Comm. in Probab.}, 15:\penalty0 457 -- 474, 2010.

\bibitem[D'Ovidio(2011)]{Dov4}
M.~D'Ovidio.
\newblock On the fractional counterpart of the higher-order equations.
\newblock \emph{Statistics \& Probability Letters}, 81:\penalty0 1929 -- 1939,
  2011.

\bibitem[Estrada and Kanwal(1993)]{EstKan93}
R.~Estrada and R.P. Kanwal.
\newblock {Taylor expansions for distributions}.
\newblock \emph{Math. Methods Appl. Sci.}, 16:\penalty0 297 -- 304, 1993.

\bibitem[Feller(1952)]{Feller52}
W.~Feller.
\newblock {On a generalization of Marcel Riesz' potentials and the semigroups
  generated by them}.
\newblock \emph{Communications du seminaire mathematique de universite de Lund,
  tome suppli-mentaire}, 1952.
\newblock dedie' a Marcel Riesz.

\bibitem[Feller(1971)]{Fel71}
W.~Feller.
\newblock \emph{An introduction to probability theory and its applications},
  volume~2.
\newblock Wiley, New York, 2 edition, 1971.

\bibitem[Hille and Phillips(1957)]{HillePhil57}
E.~Hille and R.S. Phillips.
\newblock { Functional analysis and semi-groups}.
\newblock \emph{AMS Colloquium Publications, American Mathematical Society},
  31:\penalty0 300 -- 327, 1957.

\bibitem[H\"ovel and Westphal(1972)]{HoWe72}
H.W. H\"ovel and U.~Westphal.
\newblock Fractional powers of closed operators.
\newblock \emph{Studia Math.}, 42:\penalty0 177 -- 194, 1972.

\bibitem[James(2010)]{Lanc}
L.~F. James.
\newblock Lamperti type laws.
\newblock \emph{Ann. App. Probab.}, 20:\penalty0 1303 -- 1340, 2010.

\bibitem[Komatsu(1966)]{Kom66}
H.~Komatsu.
\newblock Fractional powers of operators.
\newblock \emph{Pacific J. Math.}, 19:\penalty0 285 -- 346, 1966.

\bibitem[Krasnosel'skii and Sobolevskii(1959)]{KraSob59}
M.~A. Krasnosel'skii and P.~E. Sobolevskii.
\newblock {Fractional powers of operators acting in Banach spaces}.
\newblock \emph{Doklady Akad. Nauk SSSR}, 129:\penalty0 499 -- 502, 1959.

\bibitem[Lebedev(1972)]{LE}
N.~N. Lebedev.
\newblock \emph{{Special functions and their applications}}.
\newblock Dover, New York, 1972.

\bibitem[Meerschaert and Scheffler(2004)]{MSheff04}
M.M. Meerschaert and H.~P. Scheffler.
\newblock Limit theorems for continuous time random walks with infinite mean
  waiting times.
\newblock \emph{J. Appl. Probab.}, 41:\penalty0 623 -- 638, 2004.

\bibitem[Meerschaert et~al.(1999)Meerschaert, Benson, and Baeumer]{MBB99}
M.M. Meerschaert, D.A. Benson, and B.~Baeumer.
\newblock Multidimensional advection and fractional dispersion.
\newblock \emph{Phys. Rev. E}, 59:\penalty0 5026 -- 5028, 1999.

\bibitem[Meerschaert et~al.(2004)Meerschaert, Mortensen, and
  Scheffler]{MerMorSch04}
M.M. Meerschaert, J.~Mortensen, and H.P. Scheffler.
\newblock Vector gr\"{u}nwald formula for fractional derivatives.
\newblock \emph{Frac. Calc. Appl. Anal.}, 7:\penalty0 61 -- 81, 2004.

\bibitem[Meerschaert et~al.(2006)Meerschaert, Mortensen, and
  Wheatcraft]{MeeMorWh06}
M.M. Meerschaert, J.~Mortensen, and S.W. Wheatcraft.
\newblock Fractional vector calculus for fractional advection-dispersion.
\newblock \emph{{Physica A: Statistical Mechanics and Its Applications}},
  367:\penalty0 181 -- 190, 2006.

\bibitem[Meerschaert et~al.(2009)Meerschaert, Nane, and Vellaisamy]{BMN09ann}
M.M. Meerschaert, E.~Nane, and P.~Vellaisamy.
\newblock {Fractional Cauchy problems on bounded domains}.
\newblock \emph{Ann. Probab.}, 37\penalty0 (3):\penalty0 979 -- 1007, 2009.

\bibitem[Orsingher and Beghin(2009)]{OB09}
E.~Orsingher and L.~Beghin.
\newblock {Fractional diffusion equations and processes with randomly varying
  time}.
\newblock \emph{Ann. Probab.}, 37:\penalty0 206 -- 249, 2009.

\bibitem[Renardy and Rogers(2004)]{RenRogBook}
M.~Renardy and R.~C. Rogers.
\newblock \emph{An Introduction to partial differential equations}.
\newblock Spriger-Verlag, New York, 2004.
\newblock Second Edition.

\bibitem[Samko et~al.(1993)Samko, Kilbas, and Marichev]{SKM93}
S.G. Samko, A.~A. Kilbas, and O.~I. Marichev.
\newblock \emph{Fractional Integrals and Derivatives: Theory and Applications}.
\newblock Gordon and Breach, Newark, N. J., 1993.

\bibitem[Szarek and W\c{e}drychowicz(2002)]{SzWe02}
T.~Szarek and S.~W\c{e}drychowicz.
\newblock {Markov semigroups generated by a Poisson driven differential
  equation}.
\newblock \emph{Nonlinear Analysis}, 50:\penalty0 41 -- 54, 2002.

\bibitem[Traple(1996)]{Traple96}
J.~Traple.
\newblock {Markov semigroups generated by a Poisson driven differential
  equation}.
\newblock \emph{Bull. Polish Acad. Math.}, 44:\penalty0 161 --182, 1996.

\bibitem[Watanabe(1961)]{Wata61}
J.~Watanabe.
\newblock On some properties of fractional powers of linear operators.
\newblock \emph{Proc. Japan Acad. Ser. A Math. Sci.}, 37:\penalty0 273 -- 275,
  1961.

\end{thebibliography}
\end{document}